\documentclass{elsarticle}
\usepackage{amsmath,amssymb,amsthm}
\usepackage{color}
\usepackage{graphicx}
\usepackage{comment}
\usepackage{hyperref}

\newtheorem{thm}{Theorem}[section]

\newtheorem{lem}[thm]{Lemma}
\newtheorem{prop}[thm]{Proposition}
\newtheorem{defn}[thm]{Definition}

\newtheorem{remark}[thm]{Remark}

\allowdisplaybreaks

\begin{document}
\begin{frontmatter}
\title{Time fractional stochastic differential equations driven by pure jump L\'evy noise}

\author[1]{Peixue Wu}
\ead{peixuew2@illinois.edu}

\author[2]{Zhiwei Yang}
\ead{zwyang@mail.sdu.edu.cn}

\author[3]{Hong Wang}
\ead{hwang@math.sc.edu}

\author[1]{Renming Song}
\ead{rsong@illinois.edu}

\address[1]{Department of Mathematics, University of Illinois at Urbana-Champaign, Urbana, IL 61801, USA}

\address[2]{School of Mathematics, Shandong University, Jinan, Shandong 250100, China}

\address[3]{Department of Mathematics, University of South Carolina, Columbia, South Carolina 29208, USA}

\begin{abstract}
In this paper we introduce a variable order time fractional differential equation driven by pure jump L\'evy noise, which models the motion of a particle exhibiting memory effect. 
We prove the well-posedness of this equation without assuming any integrability
condition on the initial condition and the large jump coefficient, by using a truncation argument. Under some extra conditions, we also derive some $L^p$ moment estimates  on the solutions.
As an application of moment estimates, we prove the H\"older regularity of the solutions.
\end{abstract}

\begin{keyword}
		variable-order;  time fractional stochastic differential equation; 
	 well-posedness;  moment estimates; regularity 
\end{keyword}
\end{frontmatter}
	
\section{Introduction}
\subsection{Background and outline of the paper}
Stochastic differential equations (abbreviated as SDE) are used to describe phenomena such as particle movements with a random forcing, often modeled by Brownian motion noise (continuous)\cite{Eva14,Karatzas,Oks} or L\'evy noise (jump type) \cite{App,Rong}. \\
\indent The classical Langevin equation describes the random motion of a particle
immersed in a liquid due to the interaction with the surrounding liquid molecules. 
Let $m$ be the mass of the particle, and $u(t)$ be the instantaneous velocity of the particle. Then Newton's equation of motion for the Brownian particle is given by the Langevin equation
\begin{equation}\label{LanEq}
m\frac{du}{dt} = -\gamma u + \xi(t).
\end{equation}
Here $\gamma$ is the friction coefficient per unit mass, and $\xi(t)$ is the random force that accounts for the effect of background noise and is usually described by a white noise and its correlation function satisfies
\begin{equation*}\label{Lan:e1}
\langle \xi(t) \rangle = 0,\ \langle \xi(t_1), \xi(t_2) \rangle = 2D \delta(t_1-t_2),
\end{equation*}
where $D = \gamma k_B T$, $k_B$ is the Boltzman constant and $T$ is the absolute temperature. Understanding $\xi(t)$ as $\dot{W_t}$, where $W_t$ is a Brownian motion, \eqref{LanEq} turns into the following SDE: 
\begin{equation}\label{LanEq: SDE}
du= -\frac{\gamma}{m} u(t)dt + \frac{1}{m}dW_t.
\end{equation}
The more general SDE 
\begin{equation}\label{LanEq: SDE general}
du= f(t, u(t))dt + g(t,u(t))dW_t
\end{equation}
has been studied extensively since the last century. \\
\indent However, when the particle is immersed in viscoelastic liquids, the movement of the particle is modelled as anomalous diffusion processes and \eqref{LanEq: SDE general} can not account for this type of motion. In fact, the motion of the particle is described by a generalized Langevin equations (GLE) \cite{Zwanzig}:
\begin{equation}\label{fLanEq}
\frac{du}{dt} + \int_0^t K(t-s) u(s)ds = \eta(t),
\end{equation}
where the kernel function $K(t)$ accounts for the memory effect and the random noise $\eta(t)$ is not a white noise. \\
\indent Following the principle of \eqref{fLanEq}, various generalized Langevin equations were considered in different situations. We are particularly interested in the case when the kernel can be modelled by fractional derivatives.
Fractional Gaussian noise with correlation function
\begin{equation*}\label{fLan:e1}
\langle \eta(t) \rangle = 0, \ \langle \eta(0), \eta(t) \rangle = 
\Gamma(\alpha)t^{-\alpha},\ \alpha 
\in (0,1), 
\end{equation*}
was considered in \cite{BurBar,Kouxie,Lut}.
Using the famous fluctuation-dissipation theorem which links the memory kernel $K(t)$ with the correlation function of $\eta(t)$, the fractional Langevin equation takes the form 
\begin{equation}\label{fLanEq:e2}
\frac{du}{dt} + \gamma D_t^\alpha u = \eta(t),\ \alpha \in (0,1).
\end{equation}
We refer the reader to \cite{Jeon} for a more physically satisfactory modeling, and \cite{Besalu1, Besalu2, DeyaTindel} for a more mathematically satisfactory treatment of the equation \eqref{fLanEq} when the random noise is given by a fractional Brownian motion.\\
\indent 
In general, 
the surrounding medium of the particle may change with time, which leads to the change of the fractional dimension of the media that in turn leads to the change of the order of the fractional Langevin equation via the Hurst index \cite{EmbMae,MeeSik} and yields 
a variable-order fractional Langevin equation of the form
\begin{equation}\label{fLanEq:e4}
\frac{du}{dt} + \gamma D_t^{\alpha(t)} u = \eta(t).
\end{equation}
\indent All of the above examples focus on fractional Langevin equations with (fractional) Brownian noise. However, when the surrounding medium of the particle exhibits strong heterogeneity, the particle may experience large jumps that lead to an additional L{\'e}vy driven noise \cite{BenSchMee}. Thus, understanding $\eta(t)$ as the derivative of a stochastic process $Z_t$, \eqref{fLanEq} turns into 
\begin{equation}\label{fLanEq: fbm}
u(t) = u(0)+\int_0^t \kappa(t,s) u(s)ds + Z_t
\end{equation}
where $\kappa(t,s) = -\int_s^t K(u-s)du$ is the integral kernel. \\
\indent Based on the above discussions, we consider a generalization of \eqref{fLanEq: fbm} ,
where the random noise is given by a  multiplicative L\'evy noise and the integral kernel is modeled by variable-order fractional derivatives. Our model can be expressed as
\begin{equation}\label{model}
\begin{aligned} 
du(t) & =\big(\lambda\cdot{}_{0}^{R}D_{t}^{\alpha(t)}u(t) + f(t,u(t))\big)dt + \int_{|z|<1}g\big(t,u(t-),z\big)\widetilde{N}(dt,dz) \\
&\quad  + \int_{|z| \ge 1}h\big(t,u(t-),z\big)N(dt,dz), 
\qquad t \in [0,T],\\
u(0)&=u_0,
\end{aligned}
\end{equation}
where $T>0, \lambda \in \mathbb{R}$, $_{0}^{R}D_{t}^{\alpha(t)}u(t)$ represents friction with memory effect, $f(t,u(t))$ is the external force and the random terms represent the jump type noise.  Here the variable-order Riemann-Liouville derivative is given by 
\begin{equation} \label{fraction}
_{0}^{R}D_{t}^{\alpha(t)}u(t) : =\frac{d}{dt}\int_0^t \frac{u(s)}
{\Gamma(1-\alpha(t))(t-s)^{\alpha(t)}}ds,
\end{equation}
the kernel function $\kappa(t,s)$ is given by
$$
\kappa(t,s) = \frac{\lambda}{\Gamma(1-\alpha(t))(t-s)^{\alpha(t)}},
$$
and $\alpha, f,g,h$ are functions in proper spaces which make the integrals meaningful.  One can see similar definitions of fractional derivatives in \cite{GunLi_frac,GunLi_int,MeeSik,MetKla00,Pod,ShiWang,WanZhe,ZhangKar-B17}. \\
\indent As we will see in \eqref{integral form}, our equation can be seen as a special form of stochastic Volterra equations, with no memory effect on the random noise. For properties of general stochastic Volterra equations, 
with memory effect on the random noise, we refer the reader to \cite{Brunner, Protter, Zhang} for white noise driven stochastic Volterra equations and \cite{Agram, Benth, Chong, Erika, Protter2} for L\'evy noise driven ones. \\
\indent Compared to previous work, our main novelty is that we do not require integrability conditions on the initial distribution and large jump coefficient for the well-posedness of the solution. The difficulty is that we do not have Markov property, and  thus we can not glue together different pieces easily. In \cite{Erika}, they assumed some integrability conditions on the large jump coefficient to show the well-posedness of the solution. In our case, we are able to remove those conditions based on a truncation method and some detailed analysis on the structure of the equation. To the best of our  knowledge, this has not been done in non-Markovian settings.\\
\indent In the time fractional derivative setting, the convergence rate of the iteration procedure of the solution is slower than the usual SDE, and it can be given explicitly by the fractional order of the derivative, see \eqref{convergence rate} in the proof of Theorem \ref{main:1}. As a comparison, the convergence rate of the usual SDE is given by \cite[(2.19), page 290]{Karatzas} $$||u_n-u||_{2,T} \le c\sum_{k\ge n+1} (\frac{(LT)^k}{k!})^{1/2}.$$ 
Also, the explicit relation between the regularity of the solution and the regularity of the fractional order is also derived, see Theorem \ref{main:3}.  \\

\textbf{Outline of the paper}: The rest of this paper is organized as follows: 
In the remainder of this section, we will give an illustrative example, and present some preliminaries and notations.
In Section 2, we will formulate the assumptions and main theorems. In Section 3, we prove the well-posedness of \eqref{model}. 
In fact, 
we prove a stronger result, Proposition \ref{stronger result}, which implies the well-posedness. In Section 4, we prove some moment estimates of the solution, and as an application, we prove the H\"older regularity of the solution. 
In Section 5, we discuss some further questions, including some questions which are interesting but are outside the scope of this paper. 

\subsection{An illustrative example}
Before formulating our main problem, 
we consider the following illustrative example:
$$du(t)  =\big(\lambda\cdot {}_{0}^{R}D_{0}^{\alpha(t)}u(t) + f(t,u(t))\big)dt + h(t,u(t-))\delta_{t_0}(t), t \in [0,T], u(0) = u_0 \in \mathbb{R},
$$
where $t_0 \in (0,T)$ is a fixed time and $\delta_{t_0}$ is the Dirac measure supported at $\{t_0\}$. 
The above equation can be understood as the following Volterra equation with a jump at time $t_0$:
$$
u(t) = u_0 + \lambda \int_0^t \frac{u(s)}{\Gamma(1-\alpha(t))(t-s)^{\alpha(t)}}ds + \int_0^t f(s,u(s))ds + \int_0^t h(s,u(s-))
\delta_{t_0}(ds).
$$

Due to the memory of the term $_{0}^{R}D_{t}^{\alpha(t)}u(t)$, solving the above deterministic equation is not as easy as the memoryless case. 
The reason is that after the jump, the solution not only depends on the behavior at the jump time, but also depends on the past. Thus we need a finer piecewise analysis as below:\\
If $t < t_0$, the equation 
reduces to the Volterra equation with no jump:
$$u(t) = u_0 + \lambda \int_0^t \frac{u(s)}{\Gamma(1-\alpha(t))(t-s)^{\alpha(t)}}ds + \int_0^t f(s,u(s))ds,$$
and the solution is denoted as $v_0(t), t<t_0$.\\
If $t=t_0$, the solution is given by 
$$
u(t_0) = u_0 + \lambda \int_0^{t_0} \frac{v_0(s)}{\Gamma(1-\alpha(t_0))(t_0-s)^{\alpha(t_0)}}ds + \int_0^{t_0} f(s,v_0(s))ds 
+ h(t_0,v_0(t_0-)).
$$
If $t> t_0$, the solution is given by the following equation
$$\begin{aligned}
u(t) & = u_0 + \lambda \int_0^{t_0} \frac{v_0(s)}{\Gamma(1-\alpha(t))(t-s)^{\alpha(t)}}ds + \int_0^{t_0} f(s,v_0(s))ds + h(t_0,v_0(t_0-))\\
& \quad + \lambda \int_{t_0}^t \frac{u(s)}{\Gamma(1-\alpha(t))(t-s)^{\alpha(t)}}ds + \int_{t_0}^t f(s,u(s))ds.
\end{aligned}$$
If we define $$k(t): = u_0 + \lambda \int_0^{t_0} \frac{v_0(s)}{\Gamma(1-\alpha(t))(t-s)^{\alpha(t)}}ds + \int_0^{t_0} f(s,v_0(s))ds + h(t_0,v_0(t_0-)),$$
then the solution for $t>t_0$ is given by the following Volterra equation:
\begin{equation} \label{detexample}
u(t) = k(t) + \lambda\int_{t_0}^t \frac{u(s)}{\Gamma(1-\alpha(t))(t-s)^{\alpha(t)}}ds + \int_{t_0}^t f(s,u(s))ds.
\end{equation}

Under some conditions on $k(t)$, the above equation \eqref{detexample} has a unique solution denoted as $v_1(t), t>t_0$. Thus the global solution on $[0,T]$ can be given by 
$$
u(t)=\begin{cases}
v_0(t), &t< t_0 \\
v_0(t_0-) + h(t_0,v_0(t_0-)),& t = t_0 \\
v_1(t), &t_0< t \le T.
\end{cases}
$$ 
We plot the solutions of the ordinary differential equation and the variable-order fractional differential equation with the same jump at $t=t_0=0.5$ in Figure \ref{plot_u}:
\begin{figure}[!hbt]
	\centering
	\begin{minipage}[t]{0.5\linewidth}
		\centering
		\includegraphics[width=2in]{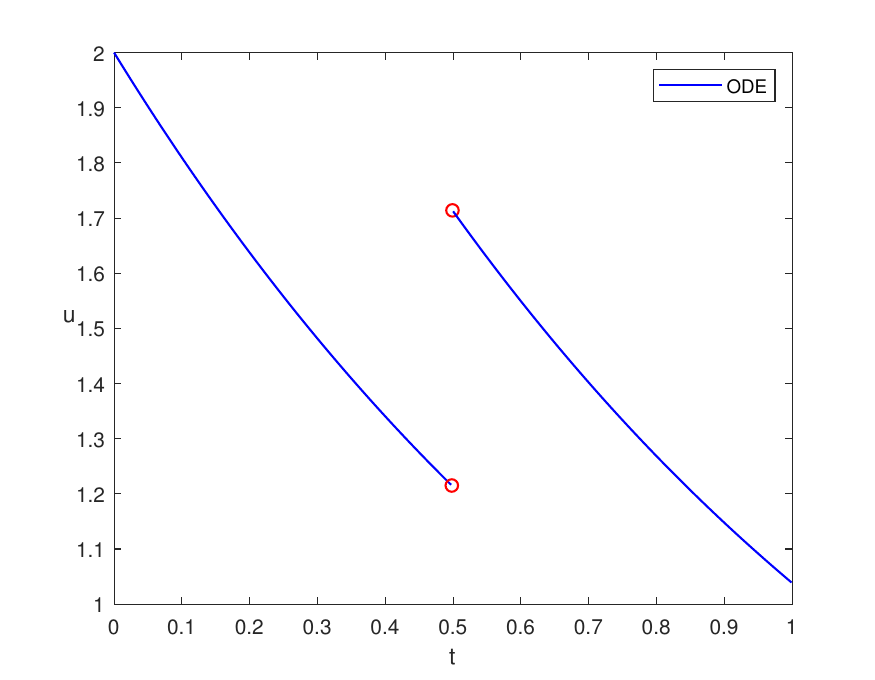}
	\end{minipage}%
	\begin{minipage}[t]{0.5\linewidth}
		\centering
		\includegraphics[width=2in]{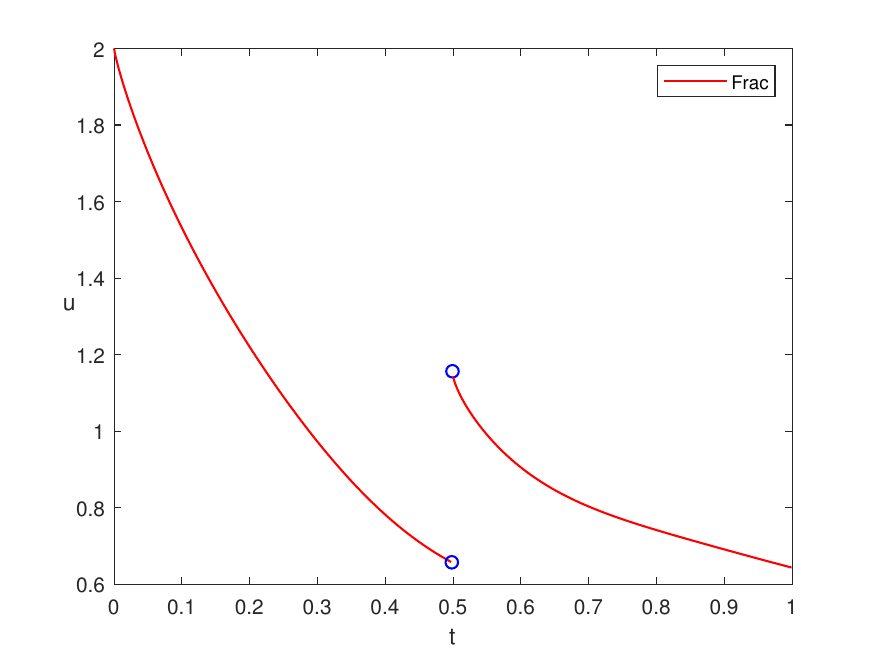}
	\end{minipage}%
	\caption{ 
		Plots of solutions: ordinary differential equation solutions ( `left one') and the variable-order fractional differential equation ( `right one') }
	\label{plot_u}
\end{figure} 

In this paper, we actually deal with a random analog of the above example. We will use a Poisson random measure to model jumps. There are two types of jumps, the first type are the ``small jumps'' and there are infinitely many of them in any finite time interval.
We deal with the small jumps by compensating the jumps 
and the procedure is similar to the white noise case. For large jumps, we use the interlacing procedure, see chapter 2 of \cite{App}. However, due to the memory term, we need a finer analysis than the usual interlacing procedure in \cite[Chapter 2]{App}.

\subsection{Preliminaries and notation}
Let $\big( \Omega, \mathcal{F}, \{\mathcal{F}_t\}_{t\ge0},\mathbb{P} \big)$ be a filtered probability space with $\{\mathcal{F}_t\}_{t\ge0}$ satisfying the usual conditions. 
Let $N(dt,dz)$ be an $\{\mathcal{F}_t\}$-adapted Poisson random measure on $\mathbb{R}_+\times \mathbb{R}^d \backslash \{0\}$ with $N(0,\cdot) = 0,\ a.s.$ and with intensity measure
$dt\nu(dx)$, where $\nu $ is a L\'evy measure which means $\nu(\{0\}) = 0$ and 
$\int_{\mathbb{R}^d} (1\wedge |x|^2) \nu(dx) < \infty.$ 
The compensated Poisson random measure is defined as
$$\widetilde{N}(dt,dz) := N(dt,dz) - \nu(dz)dt,$$ 
which is a martingale measure, see \cite[Chap. 4]{App} for the definition of martingale measure. 
We always assume the initial value $u_0\in \mathcal{F}_0$.
The argument of this paper also works for $\mathbb R^n$-valued SDEs for any $n\ge 1$. Since we are mainly interested in the interplay between the memory term and the random noise, we concentrate on the scalar case. \\
\indent 
Let $\mathbb{R}_{+}:= [0,+\infty)$ and let $B(0, 1)$ be the unit open ball in $\mathbb{R}^d$. 
The mappings $f: \mathbb{R}_+ \times \mathbb{R} \rightarrow \mathbb{R},\ g: \mathbb{R}_+ \times \mathbb{R} \times B(0, 1) \backslash \{0\} \rightarrow \mathbb{R},\ h: \mathbb{R}_+ \times \mathbb{R} \times B(0, 1)^c \rightarrow \mathbb{R} $ appearing in \eqref{model} are assumed to be measurable. They are called the the drift, small jump, large jump coefficients respectively. Sometimes we regard $g$ as a function on $\mathbb{R}_+ \times \mathbb{R} \times \mathbb{R}^d \backslash \{0\}$ by identifying it with 
$g(s,x,z) 1_{\{z: 0<|z|<1\}}(z)$, and regard $h$ as a function on $\mathbb{R}_+ \times \mathbb{R} \times \mathbb{R}^d \backslash \{0\}$ by identifying it with
$h(s,x,z) 1_{\{z: |z|\ge 1\}}(z)$.
To ensure the predictability of the jump coefficients, we always assume the following

\noindent {\bf Assumption (A0)}: For fixed $t \ge 0,z \in \mathbb{R}^d \backslash \{0\}$, 
$$
x \mapsto g(t,x,z) \mbox{  and  } x \mapsto h(t,x,z),
$$ 
are continuous functions. 

The above assumption ensures the predictability of the jump coefficients. Let $\mathcal{P} =  \{\varphi: \Omega \times [0,\infty) \times \mathbb{R}^d\backslash \{0\} \rightarrow \mathbb{R} : \varphi \ \text{predictable} \}$
be the family of all the predictable mappings. 
We refer the reader to \cite[page 216]{App} for the definition of predictable mappings.
For $1\le p \le 2$, we define
\begin{align*}
& \mathcal{L}: = \{\varphi\in \mathcal{P}: \forall t>0, \int_0^t \int_{\mathbb{R}^d\backslash \{0\}} |\varphi(s,z)|N(ds,dz) < \infty,\ a.s.\}; \\
& \mathcal{L}^p: = \{\varphi\in \mathcal{P}: \forall t>0, \mathbb{E}\int_0^t \int_{\mathbb{R}^d\backslash \{0\}} |\varphi(s,z)|^p \nu(dz)ds < \infty\}; \\
& \mathcal{L}^{p,loc}: = \{\varphi\in \mathcal{P}: \forall t>0, \int_0^t \int_{\mathbb{R}^d\backslash \{0\}} |\varphi(s,z)|^p\nu(dz)ds<\infty,\ a.s.\}.
\end{align*}
In \cite[Chap 4] {App}, stochastic integrals have been defined when the jump coefficient is in $\mathcal{L}$ and the small jump coefficient is in $\mathcal{L}^{i,loc},i=1,2$. 
For $p\in (1, 2)$, we use the standard truncation technique $\varphi = \varphi 1_{|\varphi|>1} + \varphi 1_{|\varphi|\le 1} \in \mathcal{L}^{1,loc} + \mathcal{L}^{2,loc}$ to define the integral of $\varphi$. Note that given a c\'adl\'ag process $u(t)$, $h(t,u(t-),z), g(t,u(t-),z)$ are predictable. \\
\indent For the variable-order time fractional integral operator defined by \eqref{fraction}, if we integrate both sides of \eqref{model}, we get the following integral form: for any $t\ge 0$,
\begin{equation}\label{integral form} 
\begin{split}
u(t) = &\ u_0 + \int_{0}^{t}\kappa(t,s)u(s)ds +\int_{0}^{t}f\big(s,u(s)\big)ds\\
& \int_{0}^{t}\int_{|z|<1}g\big(s,u(s-),z\big)\widetilde{N}(ds,dz)+ \int_{0}^{t}\int_{|z|\ge1}h\big(s,u(s-),z\big)
N(ds,dz),
\end{split}
\end{equation} 
where 
\begin{equation}\label{kappa}
\kappa(t,s): = \frac{\lambda}{\Gamma(1-\alpha(t))(t-s)^{\alpha(t)}}.
\end{equation}
Here is the definition of (strong) solution to \eqref{model}:
\begin{defn}
	We say an $\{\mathcal{F}_t\}-$adapted c\`adl\`ag process $\{u(t)\}_{t \ge 0}$ is a strong solution of \eqref{model} if for each $t>0$, 
	$$\begin{aligned}
	&\int_0^t |\kappa(t,s)u(s)|ds, \int_0^t |f(s,u(s))|ds,\ \int_{0}^{t}\int_{|z|<1}g\big(s,u(s-),z\big)\widetilde{N}(ds,dz),\\
	& \int_{0}^{t}\int_{|z|\ge1} |h(s,u(s-),z)|N(ds,dz),
	\end{aligned}
	$$ 
	are well-defined and finite $\mathbb{P}$-almost surely, and for each $t \ge 0$, \eqref{integral form} holds $\mathbb{P}$-almost surely.
	
	We say the strong solution of \eqref{model} is (pathwise) unique if two solutions $\tilde{u}(t), u(t)$ with $\mathbb{P}(u(0) = \tilde{u}(0)) = 1$ satisfy $$\mathbb{P}(\tilde{u}(t) = u(t), \forall t\ge0) = 1.$$
\end{defn}
Note that in the definition above, we can replace initial value $u_0$ in \eqref{integral form} by an adapted c\`adl\`ag process $\{k(t)\}$. In this case, by a strong solution to
the corresponding equation we mean an $\{\mathcal{F}_t\}-$adapted c\`adl\`ag process $\{u(t)\}_{t \ge 0}$ such that integrals in the first display in the definition above are well-defined and finite $\mathbb{P}$-almost surely, and for each $t \ge 0$, \eqref{integral form},  with $u_0$ replaced by the process $\{k(t)\}$, holds $\mathbb{P}$-a.s.. In this paper, we always deal with strong solutions in the sense above.

\vspace{.1in}

\noindent \textbf{Notational convention:} 
\begin{itemize}
\item Throughout the paper, we use capital letters $C_1(\cdot), C_2(\cdot), \cdots$ to denote constants in the statement of the results, 
the arguments inside the brackets are the parameters the constants depend on.
The lowercase letters $c_1(\cdot), c_2(\cdot), \cdots$ will denote constants used in the proof, and we do not care about their values and they may be different from one appearance to another. 
The labeling of lower case constants will start anew in every proof. 
We will only specify the dependence of the constants on $p$ and time $T>0$ and will not mention dependence on $\alpha, \lambda, f,g,h$.
\item 
The integral with respect to time $\int_0^t$ will always be understood as $\int_{(0,t]}$ in this paper.

\item For two measurable sets $A,B$ with positive measure, we write $A=B$ a.s., if there exist two null sets $N_1,N_2$ such that $A \cup N_1 = B \cup N_2$. We write  $A\subseteq B$ a.s., if there exist two null sets $N_1,N_2$ such that $A \cup N_1 \subseteq B \cup N_2$. In this paper, a.s. means $\mathbb{P}$-a.s.. We say two c\`adl\`ag  processes $u_1(t),u_2(t)$ are the same if for each $t\ge 0$, $u_1(t)=u_2(t)$ a.s.. Since the processes are c\`adl\`ag, it is equivalent to the fact that there exists a null set $N$, such that $u_1(t)=u_2(t)$ for all $t\ge0$ on $N^c$. 
\end{itemize}

\section{Assumptions and main results}
Before we give the assumptions, 
we introduce the 
definitions of $L^p$ Lipschitz continuity and $L^p$ linear growth condition for the jump coefficients. 

\begin{defn}
	Let $p>0$. We say that the small jump coefficient $g: \mathbb{R}_+ \times \mathbb{R} \times \{z\in \mathbb{R}^d: 0< |z|< 1\} \rightarrow \mathbb{R}$ is $L^p$ Lipschitz continuous if there is a  locally bounded function $L(\cdot)$ defined on $[0,\infty)$ such that for any $t\ge0, u, \tilde{u}\in \mathbb{R}$, 
	$$
	\int_{|z|< 1}|g(t,u,z)-g(t,\tilde{u},z)|^p\nu(dz)\le L(t)|u-\tilde{u}|^p.
	$$
	We say that g satisfies the $L^p$ linear growth condition if there is a  locally bounded function $L(\cdot)$ defined on $[0,\infty)$ such that for any $t\ge0, u\in \mathbb{R}$, 
	$$ 
	\int_{|z|< 1}|g(t,u,z)|^p\nu(dz)\le L(t)(1+ |u|^p).
	$$
\end{defn}

We can similarly define the $L^p$-Lipschitz continuity and $L^p$-linear growth condition for the large jump coefficient $h(t,u,z)$.
When $p=2$, the $L^p$-Lipschitz continuity and $L^p$-linear growth condition
above reduces to the usual $L^2$-Lipschitz continuity and $L^2$-linear growth condition in \cite[Chap 6]{App}.
Our assumptions on the coefficients are:\\

\noindent {\bf Assumption (A1)} (Fractional order condition): 
The fractional order function $\alpha: \mathbb{R}_{+} \rightarrow [0,\alpha^*]$ is a continuous function, where $\alpha^* \in (0,1)$.\\

\noindent {\bf Assumption (A2)} (Lipschitz condition):
 For some $p \in [1,2]$, the small jump coefficient $g$ is $L^p$-Lipschitz continuous and the drift coefficient $f$ is Lipschitz continuous, i.e., for the same locally bounded function $L(\cdot)$ above, it holds that
$$
|f(t,u) - f(t,\tilde{u})| \le L(t)|u-\tilde{u}|,\ t\ge 0, u, \tilde{u}\in\mathbb R.
$$
\noindent {\bf Assumption (A3)} (Linear growth condition):
For the same $p \in [1,2]$ as in \textbf{(A2)}, the small jump coefficient $g$ satisfies $L^p$-linear growth condition and the drift coefficient $f$ satisfies the linear growth condition, i.e., for the same locally bounded function $L(\cdot)$ above, it holds that
$$
|f(t,u)| \le L(t)(1+|u|),\ t\ge 0, u \in\mathbb R.
$$
Under the above assumptions, we can get the well-posedness of \eqref{model}:
\begin{thm} \label{main:1}
	If Assumptions \textbf{(A1)--(A3)} hold, then there exists a unique solution to \eqref{model} for any given initial distribution $u_0\in \mathcal{F}_0$.
\end{thm}	

\begin{remark}
	(A2) can be weakened to non-Lipschitz coefficients as in \cite{Wangz}. However, the trick is also based on the classical approximating procedure, which is standard so we only consider the Lipschitz case. The non-trivial part is due to the presence of the memory term. We do not need any condition on the large jump term and the initial distribution. In \cite{Erika}, they deal with the infinite-dimensional setting but assume a Lipschitz-type condition on the large jump coefficient due to the presence of the memory term.
\end{remark}

If the initial value 
$u_0$ does not have finite moment, then we can not expect the solution to have finite moment. Thus to get $L^p$ moment estimates for the solution to \eqref{model}, we need to assume $\mathbb{E}(|u_0|^p) < \infty$. 
Apart from that, we also need some assumptions on the large jump coefficients. The moment estimates of the solution to 
\eqref{model} are as follows: 

\begin{thm} \label{main:2}
Assume \textbf{(A1)} and let $u(t)$ be a solution to \eqref{model}.\\
	\textbf{Case 1:} Suppose $1 \le p \le 2$, $u_0\in L^p$ (i.e., $||u_0||_p^p: =\mathbb{E}|u_0|^p < \infty$), 
	the drift coefficient $f$ satisfies the linear growth condition and the jump coefficients $g, h$ satisfy $L^p$ linear growth conditions.
	Then for any $T>0$, we have
	\begin{equation}\label{moment const}
	\mathbb{E}\big[\sup_{0\le t \le T}|u(t)|^p\big] \le C_1(p,T,||u_0||_p)< \infty,
	\end{equation}
	where $$C_1(p,T,||u_0||_p): = c(p,T, ||u_0||_p) E_{1-\alpha^*(T),1}\big(c(p,T)\Gamma(1-\alpha^*(T))T^{1-\alpha^*(T)}\big)< \infty,$$
	and $E_{p,q}$ is the Mittag-Leffler function defined in Lemma \ref{Gronwall}.
	
	\noindent \textbf{Case 2:} 
	Suppose that $p\ge 2$, and that, in addition to the assumptions in \textbf{Case 1}, the small jump coefficient $g$ satisfies $L^2$ linear growth condition. Then \eqref{moment const} holds.
\end{thm}	

As an application of the moment estimates above, we can establish the H\"older regularity of the solution $u(t)$ to \eqref{model}. H\"older regularity means that for any  two times $t_1 \le t_2$, with $t_2 - t_1$ small, we have for some $p>0, \beta>0$, 
$$
\mathbb{E}|u(t_2)-u(t_1)|^p \le c(p)|t_2 - t_1|^{\beta}.
$$

To establish the H\"older regularity of the solution to \eqref{model}, we first need a more restrictive condition on the fractional order $\alpha(t)$:  

\noindent {\bf Assumption (A1')}:
The fractional order function $\alpha: \mathbb{R}_{+} \rightarrow (0,1)$ is locally H\"older continuous of order $\gamma>0$, i.e., for any $T>0$, there exists $C_2(T)$ such that for any $0\le t_1 < t_2 \le T$, 
$$|\alpha(t_2)-\alpha(t_1)| \le C_2(T)|t_2-t_1|^{\gamma}.$$ 

\begin{thm} \label{main:3}
	Suppose $u(t)$ is a solution to \eqref{model} and that for some $p \ge 1$
	$$
	\mathbb{E}\big[\sup_{0\le t \le T}|u(t)|^p\big] \le C_1:= C_1(p,T,||u_0||_p)< \infty.
	$$
	We further assume that \textbf{(A1')} holds, $f$ satisfies the linear growth condition and $g,h$ satisfy the $L^p$ linear growth conditions. If $p\ge 2$, we also assume that $g$  satisfies the $L^2$ linear growth condition. 
	Then for any $T>0$, 
there exist
	$C_3:=C_3(p,T)>0, C_4 = \min\{1,p\gamma, p(1-\alpha^*(T))\}>0$
	such that
	\begin{equation}
	\mathbb{E}|u(t_2)-u(t_1)|^p \le C_3|t_2 - t_1|^{C_4}, 
	0\le t_1< t_2 \le T.
	\end{equation}
\end{thm}
\begin{remark}
For the case $p\ge 2$, $\int_0^t \int_{|z|<1} g(s,u(s-),z) \widetilde{N}(ds,dz)$ may not be finite $\mathbb{P}$-a.s. if $g$ does not satisfy $L^2$ linear growth condition.
\end{remark}
\section{Well-posedness of the equation}
\subsection{Some lemmas}
Before we prove the well-posedness, we summarize some technical lemmas that will be used later in this paper. We denote by $\mathbb{D}[0,\infty)$  the space of all c\`adl\`ag functions defined on $[0,+\infty)$ with values in $\mathbb{R}$. 
The following simple lemma is easy to prove.
\begin{lem} \label{conv}
	Let $f$ be a function on $[0, \infty)$. If $f_n(\cdot)$ is a sequence of c\`adl\`ag functions such that for all $T>0$, we have
	$$\lim_{n\rightarrow \infty}\sup_{0\le t \le T}|f_n(t)-f(t)| = 0,$$ 
	then $f(\cdot) \in \mathbb{D}[0,\infty)$.
\end{lem}
\begin{lem}\label{lem:Gamma}
For any $0<a<b$, 
	there exists a constant $L := L(a,b)>0$ such that for any $x_1,x_2\in(a,b)$, 
	\begin{align}
	|  \Gamma(x_2)-\Gamma(x_1) |\leq L|x_2-x_1|.
	\end{align}
\end{lem}
\begin{proof}
This is a well-known fact. One can prove it by the mean-value theorem.
\end{proof}

\begin{lem}\label{lem:bdd}
	Under \textbf{(A1)}, there exists a constant $C_5>0$ such that for any $0\leq s<t<+\infty$,
	\begin{equation}\label{bnd:K}
	|\kappa(t,s)|\le C_5(\frac{1}{t-s} \vee 1)^{\alpha^*}.
	\end{equation}
\end{lem}
\begin{proof}
	Recalling the definition \eqref{fraction} and that $\Gamma(\cdot)$ is decreasing on $(0,1]$, we have for $t-s \le 1$, 
	$$
	|\kappa(t,s)|= \bigg|\frac{\lambda}{\Gamma(1-\alpha(t))}\bigg(\frac{1}{t-s}\bigg)^{\alpha(t)}\bigg| \le |\lambda|(t-s)^{-\alpha^*}.
	$$
	For $t-s \ge 1$, we have $|\kappa(t,s)| \le |\lambda|$. Taking $C_5 > |\lambda|$, we arrive at the desired assertion.
\end{proof}

\begin{lem}(Gronwall type inequality) \label{Gronwall} 
	Suppose $\varphi, \varphi_0, g$ are non-negative locally bounded function on $[0,\infty)$. Assume also that $g$ is non-decreasing and  $\beta \in (0,1)$ is a fixed constant. If for any $t\in [a,b)$,
	\begin{equation*}
	\varphi(t)\leq \varphi_0(t)+g(t)\int_a^t \varphi(s) (t-s)^{\beta-1}ds,
	\end{equation*}
	\vspace{-0.15in} then
	\begin{equation}\label{Gronwall:e1}
	\varphi(t)\leq \varphi_0(t)+\int_a^t \sum_{n=1}^\infty \frac{(g(t)\Gamma(\beta))^n}{\Gamma(n\beta)}(t-s)^{n\beta-1}\varphi_0(s)ds.
	\end{equation}
	Moveover, if $\varphi_0(\cdot)$ is non-decreasing, we also have
	\begin{equation}\label{Gronwall:e2}
	\varphi(t)\leq \varphi_0(t)
	E_{\beta,1}\big(g(t)\Gamma(\beta)(t-a)^\beta\big),
	\end{equation}
	where $E_{p,q}(z)$ is the Mittag-Leffler function defined by (see \cite{Die,KilSri,Pod})
	$$ E_{p,q}(z) := \sum_{k=0}^{\infty}\frac{z^k}{\Gamma(pk+q)}, \qquad z \in \mathbb{R}, ~~p \in \mathbb{R}^+, ~~q \in \mathbb{R}.$$
\end{lem}
\begin{proof}
	See \cite[Theorem 1, Corollary 2]{Ye07} for a detailed proof. One can also prove it directly by iteration and induction, similar to the argument in the proof Proposition \ref{stronger result}.
\end{proof}

Recall that $\alpha^* \in (0,1)$ is the constant in  \textbf{(A1)}.
\begin{lem} \label{scale}
If $\theta: [0,+\infty) \rightarrow \mathbb{R}$ is a locally bounded measurable, 
then for all $t>0$, 
	\begin{equation}\label{frac}
	\sup_{0\leq s\leq t}\int_0^s \frac{|\theta(r)|}{(s-r)^{\alpha^*}}dr 
	\leq \int_{0}^{t}\frac{\sup_{0\leq r\leq \tilde{u}}|\theta(r)|}{(t-\tilde{u})^{\alpha^*}}d\tilde{u}.
	\end{equation}
\end{lem}
\begin{proof}
	For any $s \in [0,t]$, by the change of variables $u = r/s$, we have
	$$
	\begin{aligned}
	\int_0^s \frac{|\theta(r)|}{(s-r)^{\alpha^*}}dr = s^{1-\alpha^*} \int_0^1 \frac{|\theta(us)|}{(1-u)^{\alpha^*}}du \le t^{1-\alpha^*} \int_0^1 \frac{\sup_{0 \le r \le ut}|\theta(r)|}{(1-u)^{\alpha^*}}du.
	\end{aligned}$$
	By the change of variables $\tilde{u} = ut$, the right-hand side of the above is equal to 
	$$ \int_0^t \frac{\sup_{0\le r \le \tilde{u}}|\theta(r)|}{(t-\tilde{u})^{\alpha^*}}d\tilde{u}.$$ 
\end{proof}
\begin{lem}(BDG-type moment identities and estimates) \label{inequality: moment}
	Fix $T>0$. Suppose $g: \Omega \times [0,T] \times \{z: 0<|z|<1\} \rightarrow \mathbb{R}$ and $h: \Omega \times [0,T] \times \{z: |z|\ge 1\} \rightarrow \mathbb{R}$ are predictable. Then we have the following: 
	\begin{enumerate}
		\item (It\^o's isometry) If for any finite stopping time $\tau$, we have 
		$$\mathbb{E} \big[\int_0^{\tau \wedge T} \int_{|z|<1} |g(s,z)|^2 \nu(dz)ds \big]< \infty,$$
		then 
		\begin{align*} 
		\mathbb{E} \Big[\int_0^{\tau \wedge T} \int_{|z|<1} |g(s,z)|^2 ds\nu(dz)\Big] = \mathbb{E} \Big[\big|\int_0^{\tau \wedge T} \int_{|z|<1} g(s,z) \widetilde{N}(ds,dz)\big|^2\Big].
		\end{align*}
		\item
		For any $p \ge 2$, if 
		$$\int_0^T \int_{|z|<1} |g(s,z)|^2 ds\nu(dz)< \infty, \quad a.s.,$$
		then there exists a $C_7(p) > 0$ such that 
		\begin{equation} \label{ge2}
		\begin{aligned}
	     & \mathbb{E}\Bigg[\sup_{0\le t \le T}\bigg|\int_0^t \int_{|z|<1} g(s,z) \widetilde{N}(ds,dz)\bigg|^p\Bigg] \\
		& \le C_7(p)\Bigg( \mathbb{E} \bigg[\big(\int_0^T \int_{|z|<1} |g(s,z)|^2 ds\nu(dz)\big)^{p/2}\bigg] + \mathbb{E} \Big[\int_0^T \int_{|z|<1} |g(s,z)|^p ds\nu(dz)\Big] \Bigg).
		\end{aligned}
		\end{equation}
		\item
		For any $p \in [1,2]$, if 
		$$\int_0^T \int_{|z|<1} |g(s,z)|^p ds\nu(dz)< \infty \quad a.s.,$$
		then there exists a $C_8(p) > 0$ such that 
		\begin{equation} \label{12}
		\mathbb{E}\big[ \sup_{0\le t \le T}\big|\int_0^t \int_{|z|<1} g(s,z) \widetilde{N}(ds,dz)\big|^p\big] \le C_8(p) \mathbb{E}\left[\int_0^T \int_{|z|<1} |g(s,z)|^p ds\nu(dz)\right]. 
		\end{equation}
		\item
		For any $p\ge 1$, if
		$$\int_0^T \int_{|z|\ge1} |h(s,z)|^p ds\nu(dz)< \infty \quad a.s..$$
		then there exists a $C_9(p,T) > 0$, which is non-decreasing in $T$, such that 
		\begin{equation}\label{large1}
		\mathbb{E} \big[ \sup_{0\le t \le T}\big|\int_0^t \int_{|z|\ge1} h(s,z) N(ds,dz)\big|^p \big] \le C_9(p,T) \mathbb{E}\left[\int_0^T \int_{|z|\ge1} |h(s,z)|^p ds\nu(dz)\right]. 
		\end{equation}
	\end{enumerate}
\end{lem}
\begin{proof}
Part 1 of Lemma \ref{inequality: moment} is proved in \cite[Theorem 4.2.3]{App}. Part 2 is proved in \cite[Theorem 4.4.23]{App}. Part 3 is proved in \cite[Theorem 1]{Novikov} by choosing $p=\alpha$ there. 
The proof of part 4 is similar to that of part 3 given in \cite{Novikov}, one only needs to replace $\{|z|<1\}$ by $\{|z|\ge 1\}$. Here is the sketch. For $1\le p \le 2$, we have
$$
\mathbb{E}\left[\sup_{0\le t \le T}|\int_0^t \int_{|z|\ge 1}h(s,z)\widetilde{N}(ds,dz)|^p\right] \le c_1(p)\int_0^T \int_{|z|\ge 1}|h(s,z)|^p \nu(dz)ds.$$
For $p \ge 2$, we have
\begin{align*}
& \mathbb{E}\left[\sup_{0\le t \le T}|\int_0^t \int_{|z|\ge 1}h(s,z)\widetilde{N}(ds,dz)|^p\right] \\
& \le c_2(p) \mathbb{E}\int_0^T \int_{|z|\ge 1}|h(s,z)|^p \nu(dz)ds+ c_2(p)\mathbb{E}\big(\int_0^T \int_{|z|\ge 1}|h(s,z)|^2 \nu(dz)ds\big)^{p/2}.
\end{align*}
Since
$N(ds,dz) = \widetilde{N}(ds,dz) + \nu(dz)ds$, we have for $1 \le p \le 2$,
$$\begin{aligned}
&\mathbb{E}\left[\sup_{0\le t \le T}|\int_0^t \int_{|z|\ge 1}h(s,z)N(ds,dz)|^p\right] \\
& \le c_3(p)\mathbb{E}\left[\int_0^T \int_{|z|\ge 1}|h(s,z)|^p \nu(dz)ds\right]+ c_3(p)\mathbb{E}\left[\int_0^T \int_{|z|\ge 1}|h(s,z)| \nu(dz)ds\right]^{p}  \\
& \le (c_4(p) + c_4(p)T^{1-\frac{1}{p}})\mathbb{E}\left[\int_0^T \int_{|z|\ge 1}|h(s,z)|^p \nu(dz)ds\right],
\end{aligned}$$
where for the last inequality, we used H\"older's inequality for the second term. For $p\ge 2$, we have similarly
$$\begin{aligned}
& \mathbb{E}\left[\sup_{0\le t \le T}|\int_0^t \int_{|z|\ge 1}h(s,z)N(ds,dz)|^p\right] \\
& \le c_5(p)(1+T^{1-\frac{1}{p}}+T^{1-\frac{2}{p}})\mathbb{E}\left[\int_0^T \int_{|z|\ge 1}|h(s,z)|^p \nu(dz)ds\right].
\end{aligned}$$
\end{proof}

\begin{lem}(Structure of large jumps, \cite[Chap. 2]{App}) \label{large}
	Suppose $h: \Omega \times [0,+\infty) \times \mathbb{R}^d\backslash \{0\} \rightarrow \mathbb{R}$ is predictable. Then for any $T>0$,
	$$
	\int_0^T \int_{|z|\ge 1}h(t,z)N(dt,dz) = \sum_{t = T_j\wedge T, j\ge 1} h(t,\Delta P(t)),
	$$
	where $P(t) := \int_0^t \int_{|z|\ge 1} zN(ds,dz)$ is 
	an $\mathbb{R}^d$-valued compound Poisson process, $\Delta P(t):= P(t)-P(t-)$ and $\{T_n\}_{n\ge 0}$ is defined by $T_0 = 0$, and for $n\ge 1$, 
	$$T_n:= \inf \{t> T_{n-1}: P(t) \neq P(t-)\}.$$
\end{lem}

\subsection{Proof of well-posedness} 
The proof of Theorem \ref{main:1} is divided into the following steps:\\
\indent \textbf{Step I}: 
First we assume the large 
jump coefficient $h(t,x,z)$ is 0, then the well-posedness
can be derived similarly as in the case of the classical fractional SDE if we assume  $\mathbb{E}[|u_0|^p] < \infty$, where 
$p \in [1,2]$ is the constant in (A2). 
Indeed, we will prove a stronger result, Proposition \ref{stronger result}, which implies the above result. 

\textbf{Step II}: We show that the moment condition for $u_0$ can be dropped using a localization trick.  

\textbf{Step III}: We use a generalized interlacing procedure to glue together the large jumps. 

Before we prove Theorem \ref{main:1}, we prove the following result which implies \textbf{Step I, Step II} and is needed in \textbf{Step III}.

\begin{prop}\label{stronger result}
	Suppose that Assumptions \textbf{(A1)--(A3)} hold and $p\in [1,2]$ is as in \textbf{(A2)}.  Suppose also that $\{k(t)\}_{t\ge 0}$ is an adapted c\`adl\`ag process such that $\forall T>0$, we have 
	$$
	\mathbb{E}(\sup_{0\le t \le T}|k(t)|^p) < \infty,
	$$
	then the following stochastic integral equation 
	\begin{equation} \label{general}
	\begin{aligned}
	u(t) &=\ k(t) + \int_{0}^{t}\kappa(t,s)u(s)ds +\int_{0}^{t}f\big(s,u(s)\big)ds\\
	&\quad + \int_{0}^{t}\int_{|z|<1}g\big(s,u(s-),z\big)\widetilde{N}(ds,dz)
	\end{aligned}
	\end{equation}
	has a unique solution. Moreover, if $k(t)\in \mathcal{F}_0,  \forall t \in [0,T]$, then the moment condition for $k(t)$ can be dropped, i.e., \eqref{general} 
	has a unique solution even if 
	$$\mathbb{E}\big[\sup_{0\le t \le T}|k(t)|^p\big] = \infty.$$
\end{prop}

\begin{proof}
	Fix an arbitrary $T>0$. \\
	\textbf{Part I}: First we show that \eqref{general} has a unique solution under the condition $
	\mathbb{E}[\sup_{0\le t \le T}|k(t)|^p] < \infty$. We define a non-decreasing function on $[0, T]$ by  
    \begin{equation}\label{initial}
    K(t,p): = \mathbb{E}\big[\sup_{0 \le s \le t}|k(s)|^p\big],\ \forall t \in [0,T].
    \end{equation} 
    For this part, we use the classical idea of proving the well-posedness of stochastic differential equations, see \cite[Chap 6]{App}, but we need to be more careful about the fractional term: \\
    \emph{1. Construction of approximating solution $u_n(t)$}.   	
	Define a sequence of approximation processes $\{u_{n}(t)\}_{n\ge 0}$ by $u_0(t) := k(t), \forall t \in [0,T]$ and for $n \ge 1$,
	\begin{equation}\label{Appr:e1}
	\begin{split}
	u_n(t) := &\ k(t) + \int_0^t \kappa(t,s) u_{n-1}(s)ds+\int_0^tf(s,u_{n-1}(s))ds \\
	& + \int_{0}^{t}\int_{|z|<1}g(s,u_{n-1}(s-),z)\tilde{N}(ds,dz).
	\end{split}
	\end{equation}
	By definition, each $u_n(t)$ is an adapted c\`adl\`ag  process on $[0,T]$. \\
	
	\noindent \emph{2. Estimating the error}.
	Our goal is to establish a good decay rate as $n\to\infty$ for 
	$$\mathbb{E}\big[\sup_{0\le s \le t} |u_n(s)-u_{n-1}(s)|^p\big].$$ 
	To establish the estimate, we apply Jensen's inequality to get 
	$$ \begin{aligned}
	&\hspace{-0.8in} \mathbb{E}\big[\sup_{0\le s \le t} |u_{n+1}(s)-u_{n}(s)|^p\big]\\
	& \hspace{-0.6in}\le c_1(p) \mathbb{E}\bigg[ \sup_{0\le s \le t}\Big( \big|\int_0^s \kappa(s,r)(u_n(r)-u_{n-1}(r))dr\big|^p \\
	&\hspace{-0.5in}+ \big|\int_0^s f(r,u_n(r))- f(r,u_{n-1}(r))dr\big|^p  \\
	&\hspace{-0.5in} + \big|\int_{0}^{s}\int_{|z|<1} (g(r,u_{n}(r-),z) - g(r,u_{n-1}(r-),z))\tilde{N}(dr,dz)\big|^p\Big)\bigg] \\
	&\hspace{-0.5in} =: (i) + (ii) + (iii).
	\end{aligned}
	$$
	For  $(i)$, we have
	$$ \begin{aligned}
	(i) &= c_1(p) \mathbb{E}\Big[\sup_{0\le s \le t} \big|\int_0^s \kappa(s,r)(u_n(r)-u_{n-1}(r))dr\big|^p\Big] \\
	&= c_1(p) \mathbb{E}\Big[\sup_{0\le s \le t} \big|\int_0^s \kappa(s,r)^{1/p'}\kappa(s,r)^{1/p}(u_n(r)-u_{n-1}(r))dr\big|^p\Big] \\
	& \le c_1(p)(\int_0^t |\kappa(t,s)| ds )^ {p/p'} \mathbb{E}\Big[ \sup_{0\le s \le t}\int_0^s |\kappa(s,r)||u_n(r)-u_{n-1}(r)|^pdr\Big] \\
	& \le c_2(p)\bigg(\int_0^t (\frac{1}{(t-s)} \vee 1 )^{\alpha^*} ds\bigg )^ {p/p'} \mathbb{E}\Big[ \sup_{0\le s \le t}\int_0^s (\frac{1}{(s-r)} \vee 1 )^{\alpha^*}|u_n(r)-u_{n-1}(r)|^p dr\Big]  \\
	& \le c_2(p)t^{\alpha^*} \bigg(\int_0^t \frac{1}{(t-s)^{\alpha^*}} ds\bigg )^ {p/p'} \mathbb{E}\Big[ \sup_{0\le s \le t}\int_0^s \frac{1}{(s-r)^{\alpha^*}} |u_n(r)-u_{n-1}(r)|^p dr\Big] \\
	& \le c_3(p)t^{\alpha^*} t^{(1-\alpha^*)p/p'} \mathbb{E}\Big[ \int_0^t  \frac{\sup_{0\le r \le s}|u_n(r)-u_{n-1}(r)|^p}{(t-s)^{\alpha^*}} ds\Big],
	\end{aligned}
	$$
	where for the first inequality, we used H\"older's inequality, and for the second inequality we used \eqref{bnd:K} and for the fourth inequality, we used \eqref{frac}.\\
	For  $(ii)$, by H\"older's inequality and the Lipschitz condition \textbf{(A2)}, we have
	$$ \begin{aligned}
	(ii) &= c_1(p) \mathbb{E}\Big[\sup_{0\le s \le t} \big|\int_0^s  f(r,u_n(r))- f(r,u_{n-1}(r)) dr\big|^p\Big] \\
	& \le c_4(p)t^{p/p'}L^*(t) \int_0^t \mathbb{E}\Big[\sup_{0 \le r \le s}|u_n(r)-u_{n-1}(r)|^p\Big] ds \\
	& \le c_4(p)t^{p/p'} t^{\alpha^*}L^*(t) \mathbb{E}\Big[ \int_0^t  \frac{\sup_{0\le r \le s}|u_n(r)-u_{n-1}(r)|^p}{(t-s)^{\alpha^*}} ds\Big],
	\end{aligned}
	$$
	where 
	\begin{equation}
	L^*(t):= \sup_{0 \le s \le t}|L(s)| < +\infty,\quad \forall t \ge 0
	\end{equation}
	is a non-decreasing function and $L(\cdot)$ is given in \textbf{(A2)}. For the second inequality above, we use the trivial bound $\frac{t}{t-s} \ge 1,\quad 0 \le s < t.$ \\
	For $(iii)$, by applying the  $L^p$ Lipschitz condition \textbf{(A2)}, we have for any $t\ge 0$
	$$
	\begin{aligned}
	& \int_0^t \int_{|z|<1} |g(r,u_{n}(s),z) - g(r,u_{n-1}(s),z)|^p\nu(dz)ds \\
	& \le L^*(t) \int_0^t
	 |u_n(s) - u_{n-1}(s)|^p ds .
	\end{aligned}
	$$ 
Since $\{u_n(t)\}_{0\le t \le T}$ is a c\`adl\`ag process, it is locally bounded a.s., consequently we have $\int_0^t |u_n(s) - u_{n-1}(s)|^p ds < \infty,\ a.s.$. 
	$$ \begin{aligned}
	(iii) &=c_1(p)\mathbb{E}\Big[\sup_{0\le s \le t} \big|\int_{0}^{s}\int_{|z|<1} (g(r,u_{n}(r-),z) - g(r,u_{n-1}(r-),z))\tilde{N}(dr,dz)\big|^p \Big]\\
	& \le c_5(p) \mathbb{E}\int_0^t \int_{|z|<1} |g(r,u_{n}(r-),z) - g(r,u_{n-1}(r-),z)|^p\nu(dz)dr \\
	& \le c_5(p)L^*(t) \int_0^t \mathbb{E}\big[\sup_{0\le r \le s}|u_n(r)-u_{n-1}(r)|^p\big] ds \\
	& \le c_5(p)t^{\alpha^*} L^*(t) \mathbb{E}\Big[ \int_0^t  \frac{\sup_{0\le r \le s}|u_n(r)-u_{n-1}(r)|^p}{(t-s)^{\alpha^*}} ds
	\Big].
	\end{aligned}
	$$
	Note that for the first inequality we used \eqref{12}, and for the second inequality above, we used $L^p$ Lipschitz continuity of $g$ and the trivial upper bound $|u_n(s)-u_{n-1}(s)|^p \le \sup_{0\le r \le s}|u_n(r)-u_{n-1}(r)|^p$. Combining the estimates for $(i),(ii),(iii)$, we 
	get that there exists  $c_6(p)>0$ such that
	\begin{equation}\label{e:2}
	\mathbb{E}\big[\sup_{0\le s \le t} |u_{n+1}(s)-u_{n}(s)|^p\big] \le 
	c_6(p) \varphi(t) \int_0^t \frac{\mathbb{E}\big[\sup_{0 \le r \le s}|u_n(r)-u_{n-1}(r)|^p\big]}{(t-s)^{\alpha^*}}ds,
	\end{equation}
	where $\varphi(t)$ is the non-decreasing function defined by 
	\begin{equation} \label{increasing function}
	\varphi(t):= t^{\alpha^*} t^{(1-\alpha^*)p/p'} + t^{\alpha^*} t^{p/p'} L^*(t)+ t^{\alpha^*}L^*(t).
	\end{equation}
	Following the same argument for \eqref{e:2}, 
	using the linear growth condition \textbf{(A3)} instead of the Lipschitz condition, we have 
	\begin{equation}\label{e:3}
	\mathbb{E}\big[\sup_{0\le s \le t} |u_{1}(s)-u_{0}(s)|^p\big] \le c_7(p)\varphi(t) \int_0^t \frac{\mathbb{E}\big[\sup_{0 \le r \le s}|k(r)|^p\big]}{(t-s)^{\alpha^*}}ds.
	\end{equation}
	Now we define 
	$$g_n(t):= \mathbb{E}\big[\sup_{0\le s \le t}|u_{n}(s)-u_{n-1}(s)|^p\big],\quad n\ge 1.$$ 
	It follows from \eqref{e:3} that 
	\begin{equation} \label{est1}
	g_1(t)\le c_{7}(p)\varphi(t) K(t,p) \int_0^t \frac{1}{(t-s)^{\alpha^*}}ds = c_{8}(p) \varphi(t)K(t,p)t^{1-\alpha^*} \le c_9 K(T,p)t^{1-\alpha^*},
	\end{equation}
	where $c_9$ is given by 
	\begin{equation}
	c_9= c_{9}(p,T):= \max\{c_{6}(p)\varphi(T), c_{8}(p)\varphi(T)\}.
	\end{equation}
	We claim that for any $n\ge1,\ t\in [0,T]$,	
	\begin{equation} \label{est3}
	g_n(t) \le \frac{[c_{9}\Gamma(1-\alpha^*)t^{1-\alpha^*}]^n}{\Gamma(1+n(1-\alpha^*))}(1-\alpha^*)K(T,p).
	\end{equation}
	The claim is true for $n=1$ by \eqref{est1} and the identity $\Gamma(1-\alpha^*)(1-\alpha^*) = \Gamma(2-\alpha^*)$. Now we prove the claim by induction. Suppose the claim is true for $n \ge 1$. Then by \eqref{e:2}, we have 
	\begin{equation} \label{est2}
	g_{n+1}(t) \le c_9 \int_0^t \frac{g_n(s)}{(t-s)^{\alpha^*}}ds.
	\end{equation}
	Plugging in the induction step, we have 
\begin{align*}
g_{n+1}(t) & \le c_9 \frac{[c_{9}\Gamma(1-\alpha^*)]^n}{\Gamma(1+n(1-\alpha^*))}(1-\alpha^*)K(T,p)\int_0^t \frac{s^{n(1-\alpha^*)}}{(t-s)^{\alpha^*}}ds \\
& = \frac{[c_{9}\Gamma(1-\alpha^*)t^{1-\alpha^*}]^{n+1}}{\Gamma(1+(n+1)(1-\alpha^*))}(1-\alpha^*)K(T,p).
\end{align*}	
	For the last equality, we used change of variable $u = s/t$, the definition of beta function defined by $B(x,y):=\int_0^1 s^{x-1} (1-s)^{y-1}ds,\ x,y>0$ and the famous quantity $B(x,y) = \frac{\Gamma(x)\Gamma(y)}{\Gamma(x+y)}$ relating the gamma function and beta function.  \\
	
	\noindent \emph{3. Probabilistic argument to show $u_n(t)$ tends to the unique solution $u(t)$}. 
	By Chebyshev's inequality, we have 
	$$\begin{aligned}
	&\sum_{n\ge 1}\mathbb{P}(\sup_{0\le s \le t}|u_n(s)-u_{n-1}(s)|> 2^{-n}) \\
	& \le \sum_{n\ge 1} \frac{[2^p c_{9} \Gamma(1-\alpha^*)t^{1-\alpha^*}]^n}{\Gamma(1+n(1-\alpha^*))}(1-\alpha^*)K(T,p) < \infty.
	\end{aligned}$$
	The finiteness follows from the well-known asymptotic behavior of $\Gamma$: 
	$$\Gamma(1+n(1-\alpha^*)) \sim \sqrt{n(1-\alpha^*)}\big(\frac{n(1-\alpha^*)}{e}\big)^{n(1-\alpha^*)}.$$
	It follows from the Borel-Cantelli lemma and a standard limiting argument, see for example \cite[Chap 6]{App}, that there exists an adapted c\`adl\`ag process $u(t), 0\le t \le T$, such that 
	$$\lim_{n\rightarrow \infty} \sup_{0\le t \le T}|u_n(t)-u(t)| = 0, \ a.s..$$
	To show that $u(t)$ is actually a solution to \eqref{general} on $[0,T]$, we define for any $t \in [0,T]$, 
	$$\tilde{u}(t): = k(t) + \int_{0}^{t}\kappa(t,s)u(s)ds +\int_{0}^{t}f\big(s,u(s)\big)ds +
	\int_{0}^{t}\int_{|z|<1}g\big(s,u(s-),z\big)\widetilde{N}(ds,dz).$$
Our aim is to show $\sup_{0\le t \le T}|\tilde{u}(t)- u(t)| = 0, a.s.$, which implies that $u(t)$ is a strong solution to \eqref{general}. To this end, it is sufficient to show that $u_n(t)$ converges to $u(t),\tilde{u}(t)$ simultaneously in $L^p(\Omega;L^{\infty}([0,T];\mathbb{R}))$ with the norm defined by
	$$||v||_{p,T}:= \big(\mathbb{E}[\sup_{0\le t \le T}|v(t)|^p]\big)^{1/p}.$$
	It is easy to check that $L^p(\Omega;L^{\infty}([0,T];\mathbb{R}))$ is a Banach space. By \eqref{est3}, $\{u_n(t)\}$ forms a Cauchy sequence in $L^p([0,T])$ because $$\sum_{n\ge1} \bigg(\frac{[c_{9} \Gamma(1-\alpha^*)T^{1-\alpha^*}]^n}{\Gamma(1+n(1-\alpha^*))}\bigg)^{1/p} < \infty.$$
	The limit coincides with the a.s. limit $u(t)$ and the convergence rate is given by
	\begin{equation} \label{convergence rate}
	||u_n-u||_{p,T} \le \sum_{k \ge n+1}  \bigg(\frac{[c_{9}\Gamma(1-\alpha^*)T^{1-\alpha^*}]^k}{\Gamma(1+k(1-\alpha^*))} (1-\alpha^*)K(T,p)\bigg)^{1/p},
	\end{equation}
	which tends to zero as $n\to \infty$ using the asymptotic behavior of $\Gamma.$ To show $||u_n-\tilde{u}||_{p,T} \rightarrow 0$, using the same argument for \eqref{e:2}, we get
	$$||u_n-\tilde{u}||_{p,T}^p = \mathbb{E}\big[\sup_{0\le t \le T}|u_n(t)-\tilde{u}(t)|^p\big] \le c_{10}(p)\varphi(T) \int_0^T \frac{||u_{n-1}-u||_{p,T}^p}{(T-t)^{\alpha^*}}dt.$$
	The above tends to zero by \eqref{convergence rate} as $n\to\infty$. In summary, we showed that 
$$\lim_{n\to \infty}||u_n-u||_{p,T} = \lim_{n\to \infty}||u_n-\tilde{u}||_{p,T} = 0,$$	
	thus $\sup_{0\le t \le T}|\tilde{u}(t)- u(t)| = 0, a.s.$, which implies that $u(t)$ is a strong solution to \eqref{general}. \\
	\indent 
	To finish the proof of the first part, it remains to show that if $u,v$ are two strong solutions with the same process $\{k(t)\}$, then they must be the same. 
	Let $h(t):=||u-v||_{p,t}^p $, 
	then using the same argument for \eqref{e:2}, we get
	$$h(T)=||u-v||_{p,T}^p \le c_{11}(p)\varphi(T)\int_0^T\frac{h(t)}{(T-t)^{\alpha^*}}dt.$$
	Recalling $\varphi$ is a non-decreasing function defined in \eqref{increasing function}, we can apply the generalized Gronwall's inequality \eqref{Gronwall:e2} such that $h(T) = 0$, which implies   
	$$\sup_{0\le t \le T}|u(t)- v(t)| = 0, a.s.,$$
i.e., the solution is unique. The first part is complete. \\
	
	\noindent \textbf{Part II}: $K(T,p): =\mathbb{E}\sup_{0 \le t \le T}|k(t)|^p = \infty$, but we have $k(t) \in \mathcal{F}_0, \forall t \ge 0.$\\
	Define
	$$\Omega_N:= \{\sup_{0 \le t \le T}|k(t)| \le N\} \in \mathcal{F}_0, \quad N\ge 1.$$
	We have 
	$$\Omega = \bigcup_{N\ge 1}\Omega_N, \ a.s.$$ Indeed, since $k(t)$ is a c\`adl\`ag process, 
	there exists a $\mathbb{P}$-null set $\mathcal{N}$ such that, when $\omega \in \mathcal{N}^c$, $k(t)(\omega)$ is a c\`adl\`ag function on $[0,T]$, which in particular is bounded. Therefore we must have $\omega$ is in some $\Omega_N$, which implies $\Omega \subset \bigcup_{N\ge 1}\Omega_N, \ a.s.$. The other direction is obvious. \\
	\indent From \textbf{Part I}, we know that for any $N \ge 1$, the following stochastic integral equation 
	\begin{equation}\label{localequation}
	\begin{aligned}
	w(t)& = k(t)1_{\Omega_N} + \int_{0}^{t}\kappa(t,s)w(s)ds +\int_{0}^{t}f\big(s,w(s)\big)ds\\
	&\quad+ \int_{0}^{t}\int_{|z|<1}g\big(s,w(s-),z\big)
	\widetilde{N}(ds,dz) \end{aligned}
	\end{equation}
	has a unique solution which is c\`adl\`ag and adapted, and we denote it as $u^N(t), t\in [0,T]$. Thus we can define for all $t \in [0,T]$,
    \begin{equation} \label{construction of solution}
     u(t)(\omega): = u^N(t)(\omega),\quad \text{if}\ \omega \in \Omega_N. 
    \end{equation}    	
	 First we need to show that $u(t)$ is well-defined. 
	It is equivalent to show that if $M>N$, for $\mathbb{P}$-almost all $ \omega \in \Omega_N \subset \Omega_M$, we have $u^N(t)(\omega) = u^M(t)(\omega),\ \forall t\in [0,T]$. In fact, by Jensen's inequality, we have
	$$\begin{aligned}
	& \mathbb{E}\big[\sup_{0\le t \le T}|u^N(t)-u^M(t)|^p 1_{\Omega_N}\big]\\
	& \le c_1(p)\mathbb{E}\bigg[\sup_{0\le t \le T}\Big(   \big|\int_{0}^{t}\kappa(t,s)(u^N(s)-u^M(s))ds\big|^p 1_{\Omega_N} \\ 
	&\quad+\big|\int_{0}^{t}\big(f(s,u^N(s)) - f(s,u^M(s))\big)ds\big|^p 1_{\Omega_N} \\
	& \quad+ \big|\int_{0}^{t}\int_{|z|<1} \big(g(s,u^N(s-),z) - g(s,u^M(s-),z)\big)\widetilde{N}(ds,dz)\big|^p 1_{\Omega_N}\Big) \bigg].
	\end{aligned}$$
	For the first two integrals, the indicator function $1_{\Omega_N}$ can be put inside because the integral is pathwisely defined. To put the indicator function inside the stochastic integral, we introduce the following stopping time (it is a stopping time because $k(t) \in \mathcal{F}_0,\ \forall t \ge 0$)$$T_N(\omega) = \begin{cases}
	0, &\omega \notin \Omega_N, \\
	\infty,& \omega \in \Omega_N.
	\end{cases}$$
	Then we have $\forall t \in (0,T]$, 
	\begin{equation} \label{localization}
	1_{\Omega_N} = 1_{\{T_N = \infty\}} = 1_{\{T_N \ge t\}}.
	\end{equation}
	
	Now we use the following lemma, whose proof will be given later, to put the indicator function inside the stochastic integral: 
	\begin{lem} \label{l:3.10}
		Suppose $g$ is in $\mathcal{L}^{p,loc}$ for the above $p \in [1,2]$, i.e., 
		$$\int_0^T \int_{\{z: |z|<1\}}|g(s,z)|^p \nu(dz)ds<\infty,\ a.s.$$
		Then we have for any $t \in [0,T]$,
		\begin{equation}\label{inside}
		\int_{0}^{t}\int_{|z|<1} g(s,z)\widetilde{N}(ds,dz) 1_{\Omega_N} = \int_{0}^{t}\int_{|z|<1} g(s,z) 1_{\Omega_N}\widetilde{N}(ds,dz).
		\end{equation}
	\end{lem}	
	
	\noindent Using the $L^p$ Lipschitz continuity, it is easy to check that the predictable process $g(s,z):= g(s,u^N(s-),z) - g(s,u^M(s-),z)$ satisfies the assumption of Lemma \ref{l:3.10}. Therefore,
	$$\begin{aligned}
	& \big|\int_{0}^{t}\int_{|z|<1} \big(g(s,u^N(s-),z) - g(s,u^M(s-),z)\big)\widetilde{N}(ds,dz)\big|^p 1_{\Omega_N} \\
	&= \big|\int_{0}^{t}\int_{|z|<1} \big(g(s,u^N(s-),z) - g(s,u^M(s-),z)\big)1_{\Omega_N}\widetilde{N}(ds,dz)\big|^p.
	\end{aligned}$$
Using the same argument for \eqref{e:2}, we get
	$$
	\mathbb{E}\Big[\sup_{0\le t \le T}|u^N(t)-u^M(t)|^p 1_{\Omega_N} \Big] \le c_{12}(p)\varphi(T) \int_0^T \frac{\sup_{0\le r \le s}\mathbb{E}\big[|u^N(r)-u^M(r)|^p 1_{\Omega_N}\big]}{(T-s)^{\alpha^*}} ds.
	$$
	Thus by the generalized Gronwall's inequality \eqref{Gronwall:e2} we have 
	$$
	\sup_{0 \le t \le T}\big|u^N(t) - u^M(t)\big| 1_{\Omega_N} = 0, \ a.s.,
	$$
	which shows that $u(t)$ is well-defined.
	
	It remains to show that $u(t)$ is the unique solution to \eqref{general}. First we show that it is actually a solution to \eqref{general}. Indeed, for any $N\ge1$, using \eqref{inside} twice, we have for any $t \in [0,T]$,
	\begin{align*}
	u(t)1_{\Omega_N} =&\ u^N(t)1_{\Omega_N}  = k(t)1_{\Omega_N} + \big\{\int_{0}^{t}\kappa(t,s)u^N(s)ds +\int_{0}^{t}f\big(s,u^N(s)\big)ds \\
	& +\int_{0}^{t}\int_{|z|<1}g\big(s,u^N(s-),z\big)\widetilde{N}(ds,dz)\big\}1_{\Omega_N} \\
	=& \ k(t)1_{\Omega_N} + \int_{0}^{t}\kappa(t,s)u^N(s)1_{\Omega_N}ds +\int_{0}^{t}f\big(s,u^N(s)\big)1_{\Omega_N}ds \\
	& +\int_{0}^{t}\int_{|z|<1}g\big(s,u^N(s-),z\big)1_{\Omega_N}\widetilde{N}(ds,dz)\\
	=& \ \Big[ k(t) + \int_{0}^{t}\kappa(t,s)u(s)ds +\int_{0}^{t}f\big(s,u(s)\big)ds \\
	& \ +\int_{0}^{t}\int_{|z|<1}g\big(s,u(s-),z\big)\widetilde{N}(ds,dz)\Big]1_{\Omega_N}. 
	\end{align*}
	Thus \eqref{general} holds on every $\Omega_N$, thus on $\Omega = \bigcup_{N\ge1}\Omega_N\ a.s.$, which implies that $u(t)$ is actually a solution.
	
	\textbf{Step II} would be finished if we can show the solution of \eqref{general} is unique. Indeed, suppose $v(t)$ is also a solution to \eqref{general} with the same $\{k(t)\}$, then we claim that for all $N \ge 1$ and any $t\in [0,T]$, $v(t) = u^N(t)$, a.s. on $\Omega_N$. If the claim is true, recalling the definition of $u(t)$ in \eqref{construction of solution}, we know that $v(t)$ coincides with $u(t)$ a.s. and the proof would be complete.
	
	Now we prove the above claim by contradiction: suppose the claim is not true. In other words, there exists an $N_0\ge 1$, such that $v(t)$ and $u^{N_0}(t)$ are not the same on $\Omega_{N_0}$. In other words, there exists a measurable set $\mathcal{N} \subseteq \Omega_{N_0}$ a.s. with $\mathbb{P}(\mathcal{N})>0$ such that for any $\omega_0 \in \mathcal{N}$, there exists a $t_0 = t_0(\omega_0) \in [0,T]$, $v(t_0)(\omega_0) \neq u^{N_0}(t_0)(\omega_0)$. To get the contradiction, we construct another c\`adl\`ag process $\widetilde{v}(t)$ defined by 
	\begin{equation} \label{constructed}
	\widetilde{v}(t) = v(t)1_{\Omega_{N_0}} + u^{N_0}(t)1_{\Omega_{N_0}^c}.
	\end{equation}
	Note that $\widetilde{v}(t)$ is adapted because $1_{\Omega_N} \in \mathcal{F}_0$. 
	We claim that $\widetilde{v}(t)$ is a solution to \eqref{localequation}. However, since $v(t)$ and $u^{N_0}(t)$ are not the same on $\Omega_{N_0}$, the claim contradicts the fact that \eqref{localequation} has only one solution. This will prove the uniqueness of \eqref{general}. 
	
	To show that $\widetilde{v}(t)$ is a solution of \eqref{general}, note that $v(t)$ is a solution of \eqref{general} and $u^{N_0}(t)$ is a solution of \eqref{localequation}. Therefore, we have 
	\begin{align*}
	\widetilde{v}(t) & = v(t)1_{\Omega_{N_0}} + u^{N_0}(t)1_{\Omega_{N_0}^c} \\
	& = k(t)1_{\Omega_{N_0}} + \big\{\int_{0}^{t}\kappa(t,s)v(s)ds +\int_{0}^{t}f\big(s,v(s)\big)ds \\
	& +\int_{0}^{t}\int_{|z|<1}g\big(s,v(s-),z\big)\widetilde{N}(ds,dz)\big\}1_{\Omega_{N_0}} \\
	& + k(t)1_{\Omega_{N_0}}1_{\Omega_{N_0}^c} + \big\{\int_{0}^{t}\kappa(t,s)u^{N_0}(s)ds +\int_{0}^{t}f\big(s,u^{N_0}(s)\big)ds \\
	& +\int_{0}^{t}\int_{|z|<1}g\big(s,u^{N_0}(s-),z\big)\widetilde{N}(ds,dz)\big\}1_{\Omega_{N_0}^c} \\
	& = k(t)1_{\Omega_{N_0}} + \int_{0}^{t}\kappa(t,s)( v(s)1_{\Omega_{N_0}} + u^{N_0}(s)1_{\Omega_{N_0}^c})ds \\
	& + \int_{0}^{t} \big( f\big(s,v(s)\big)1_{\Omega_{N_0}} + f\big(s,u^{N_0}(s)\big)1_{\Omega_{N_0}^c} \big)ds \\
	& + \int_{0}^{t}\int_{|z|<1} \big(g\big(s,v(s-),z\big)1_{\Omega_{N_0}} + g\big(s,u^{N_0}(s-),z\big)1_{\Omega_{N_0}^c}\big)\widetilde{N}(ds,dz)\big\} \\
	& = k(t)1_{\Omega_{N_0}} + \int_{0}^{t}\kappa(t,s)\widetilde{v}(s)ds + \int_{0}^{t} f\big(s,\widetilde{v}(s)\big) ds \\
	& + \int_{0}^{t}\int_{|z|<1} g\big(s,\widetilde{v}(s-),z\big) \widetilde{N}(ds,dz),
	\end{align*}
	which shows that $\widetilde{v}(t)$ is a solution of \eqref{localequation}, and the uniqueness of \eqref{general} is proved. The proof of the proposition is complete.
\end{proof}

\textbf{Proof of Lemma \ref{l:3.10}}. We only need to deal with the case $t>0$. 
First we use the standard approximation in \cite[Chap 4]{App}: 
suppose $A_n$ increases to $\{z: |z|<1\}$, and $\nu(A_n)< \infty, \forall n \ge 1$. Then we have for any $t \in [0,T]$,
$$\int_{0}^{t}\int_{A_n} g(s,z)\widetilde{N}(ds,dz) 1_{\Omega_N} \rightarrow \int_{0}^{t}\int_{|z|<1} g(s,z)\widetilde{N}(ds,dz) 1_{\Omega_N}$$
in probability. We only need to show the following for any $n\ge 1$:
\begin{equation}\label{inside2}
\int_{0}^{t}\int_{A_n} g(s,z)\widetilde{N}(ds,dz) 1_{\Omega_N} = \int_{0}^{t}\int_{A_n} g(s,z) 1_{\Omega_N}\widetilde{N}(ds,dz).
\end{equation}
To prove \eqref{inside2}, we use the definition of the integral. Recalling that $\nu(A_n)< \infty$ and applying Lemma \ref{large}, 
we can define $P^n(t) = \int_0^t \int_{A_n}z N(ds,dz)$, 
which is a compound Poisson process. Define $\{T_j^n\}_{j\ge 1}$ by $T_0^n = 0$ and for $j\ge 1$, 
$$T_j^n := \inf \{t > T_{j-1}^n: P^n(t) \neq P^n(t-)\}.$$
Then the left-hand side of \eqref{inside2} is equal to 
\begin{align*}
& \int_{0}^{t}\int_{A_n} g(s,z)\widetilde{N}(ds,dz) 1_{\Omega_N} \\
& = \big[\sum_{s \le t\wedge T_j^n, j\ge 1} g(s,\Delta P(s))\big] 1_{\Omega_N} - \int_0^t \int_{A_n}g(s,z)\nu(dz)ds1_{\Omega_N} \\
& = \big[\sum_{s \le t\wedge T_j^n, j\ge 1} g(s,\Delta P(s))\big] 1_{\{T_N \ge t\}} - \int_0^t \int_{A_n}g(s,z)1_{\Omega_N}\nu(dz)ds \\
& = \sum_{s \le t\wedge T_j^n \wedge T_N, j\ge 1} g(s,\Delta P(s)) - \int_0^t \int_{A_n}g(s,z)1_{\Omega_N}\nu(dz)ds \\
& = \sum_{s \le t\wedge T_j^n, j\ge 1} \big[g(s,\Delta P(s))1_{T_N \ge s}\big] - \int_0^t \int_{A_n}g(s,z)1_{\Omega_N}\nu(dz)ds \\
& = \sum_{s \le t\wedge T_j^n, j\ge 1} \big[g(s,\Delta P(s))1_{\Omega_N}\big] - \int_0^t \int_{A_n}g(s,z)1_{\Omega_N}\nu(dz)ds \\
& =\int_{0}^{t}\int_{A_n} g(s,z) 1_{\Omega_N}\widetilde{N}(ds,dz).
\end{align*}
For the second and fifth equality, we used \eqref{localization}. Also note that $1_{\Omega_N} \in \mathcal{F}_0$ thus $g(s,z) 1_{\Omega_N}$ is predictable. Thus \eqref{inside2} is true, and the proof of the lemma is complete.
\qed

\noindent {\textbf{Proof of Theorem \ref{main:1}}}. 
Taking $k(t) \equiv u_0 $ in Proposition \ref{stronger result}, we immediately get \textbf{Step I} and {\textbf Step II}. \\
\indent \textbf{Step III}: Suppose $h \neq 0$. \\
We use the stopping times $\{T_n\}_{n\ge 0}$ in Lemma \ref{large} to construct the solution inductively. From \textbf{Step II}, define $v_0(t)$ to be the unique solution to 
$$v(t) = u_0 + \int_0^t \kappa(t,s)v(s)ds + \int_0^t f(s,v(s))ds + \int_0^t \int_{|z|<1}g(s,v(s-),z)\widetilde{N}(ds,dz),$$ 
then we define $u(t), t\in [0,T_1]$ as
$$u(t) = \begin{cases}
v_0(t), &t < T_1,\\
v_0(T_1-) + h(T_1, v_0(T_1-),\Delta P(T_1)), &t = T_1.
\end{cases}$$
After the first jump, the memory term comes into play. Since 
$$T_1, P(T_1) \in \mathcal{F}_{T_1},\quad \sigma\{v_0(t): 0\le t < T_1\} \subset \mathcal{F}_{T_1},$$
we know that
\begin{align*}
\hspace{-0.5in}k_1(t)&: = u_0 + \int_0^{T_1}\kappa(t,s)v_0(s)ds + \int_0^{T_1} f(s,u(s))ds \\
& + \int_0^{T_1} \int_{|z|<1}g(s,u(s-),z)\widetilde{N}(ds,dz) + h(T_1, v_0(T_1-), \Delta P(T_1))
\end{align*}
is $\mathcal{F}_{T_1}$-measurable for any $t \ge T_1$. Then applying Proposition \ref{stronger result}, replacing $\mathcal{F}_0$ by $\mathcal{F}_{T_1}$ in our case, the following equation
$$v(t) = k_1(t)+ \int_{T_1}^t \kappa(t,s)v(s)ds + \int_{T_1}^t f(s,v(s))ds + \int_{T_1}^t \int_{|z|<1}g(s,v(s-),z)\widetilde{N}(ds,dz)$$
has a unique solution $v_1(t)$ on $[T_1, \infty)$. Define $u(t)$ on $[T_1,T_2]$ by
$$u(t) = \begin{cases}
v_1(t),& T_1 \le t < T_2, \\
v_1(T_2-) + h(T_2, v_1(T_2-),\Delta P(T_2)), &t = T_2.
\end{cases}$$
Suppose we already defined the solution $u(t)$ on $[0,T_n], n \ge 2$, and $v_i(t), 1 \le i \le n-1$, then we define $$\begin{aligned}
k_n(t): =& u_0 + \int_0^{T_n}\kappa(t,s)u(s)ds + \int_0^{T_n} f(s,u(s))ds \\
& + \int_0^{T_n} \int_{|z|<1}g(s,u(s-),z)\widetilde{N}(ds,dz) + \sum_{i=1}^n h(T_i, v_{i-1}(T_i-), \Delta P(T_i)).
\end{aligned}$$
The following equation
$$v(t) = k_n(t)+ \int_{T_n}^t \kappa(t,s)v(s)ds + \int_{T_n}^t f(s,v(s))ds + \int_{T_n}^t \int_{|z|<1}g(s,v(s-),z)\widetilde{N}(ds,dz)$$
has a unique solution $v_n(t)$ on $[T_n,\infty)$. Now we can define $u(t)$ on $[T_n,T_{n+1}]$ by
$$u(t) = \begin{cases}
v_{n}(t), &T_{n} \le t < T_{n+1},\\
v_{n}((T_{n+1}-) + h(T_{n+1}, v_{n}(T_{n+1}-),\Delta P(T_{n+1})), &t = T_{n+1}.
\end{cases}$$
Since $\lim_{n \rightarrow \infty}T_n = \infty,\ a.s.$, the solution is defined on the whole $[0,\infty)$. \\
\indent Now from \textbf{Step II} and Lemma \ref{large}, $u(t)$ is indeed a solution to \eqref{model} and piecewisely unique, thus unique on the whole $[0,\infty)$. \qed 

\section{Proof of moment estimates and H\"older regularity}
\noindent \textbf{Proof of Theorem \ref{main:2}}: 
Recall that $u(t)$ is a solution to \eqref{model}. Fix $T>0$. \\
If $1\le p\le 2$, by the linear growth condition of $f,g,h$, we can apply the inequalities  \eqref{12} and \eqref{large1}. Then by Jensen's inequality and the same argument for \eqref{est2}, we get the following moment estimate:
$$
\begin{aligned}
& \mathbb{E}\Big[ \sup_{0\le t \le T}|u(t)|^p \Big] \\
& \le c_1(p) \bigg \{ \mathbb{E}\big[|u_0|^p\big] + \mathbb{E}\Big[ \sup_{0\le t \le T}|\int_{0}^{t}\kappa(t,s)u(s)ds|^p\Big] + \mathbb{E}\Big[ \sup_{0\le t \le T}|\int_{0}^{t}f\big(s,u(s)\big)ds|^p \Big]\\
&+ \mathbb{E}\Big[ \sup_{0\le t \le T}\big|\int_{0}^{t}\int_{|z|<1}g\big(s,u(s-),z\big)\widetilde{N}(ds,dz)\big|^p  \Big]\\
&+ \mathbb{E}\Big[ \sup_{0\le t \le T}\big|\int_{0}^{t}\int_{|z|\ge1}h\big(s,u(s-),z\big)N(ds,dz)\big|^p\Big]
\bigg\} \\
& \le c_2(p,T) \bigg(\mathbb{E}|u_0|^p + 1 + \int_0^T\frac{\mathbb{E}\big[\sup_{0\le s \le t}|u(s)|^p\big]}{(T-t)^{\alpha^*}}dt\bigg),
\end{aligned}$$
where $c_2(p,T)$ is non-decreasing in $T$ for fixed $p$. The reason is that $c_2(p,T)$ is defined via the sum of $\varphi(T)$ and $C_9(p,T)$. Recall that $\varphi(T)$ is defined in \eqref{increasing function} and $C_9(p,T)$ is chosen in Lemma \ref{inequality: moment}. Then by the generalized Gronwall inequality, we have 
$$\mathbb{E}\big[\sup_{0\le t \le T}|u(t)|^p\big] \le c_3(p,T,||u_0||_p) E_{1-\alpha^*,1}(c_2(p,T)\Gamma(1-\alpha^*)T^{1-\alpha^*})< \infty,$$
where $E_{p,q}$ is the Mittag-Leffler function defined in Lemma \ref{Gronwall}. \\
\noindent For the case $p\ge 2$, since we assume in addition that $g$ satisfies $L^2$ linear growth condition, we have the desired result by applying \eqref{ge2} instead of \eqref{12}. 
\qed  \\

\noindent \textbf{Proof of Theorem \ref{main:3}}: 
First we assume $1 \le p \le 2$. 
Since $u(t)$ is a solution to \eqref{model}, for $0\le t_1 < t_2 \le T$, we have
\begin{align*}
u(t_2)- u(t_1) & = \int_0^{t_2}\kappa(t_2,s)u(s)ds - \int_0^{t_1}\kappa(t_1,s)u(s)ds + \int_{t_1}^{t_2} f(s,u(s))ds  \\
& + \int_{t_1}^{t_2} \int_{\{z: |z|<1\}}g(s,u(s))\widetilde{N}(ds,dz) + \int_{t_1}^{t_2} \int_{\{z: |z|\ge 1\}}h(s,u(s))N(ds,dz).
\end{align*}
Then using the argument for \eqref{est2} and applying Theorem \ref{main:2}, we get
\begin{align*}
& \mathbb{E}\bigg|\int_{t_1}^{t_2} f(s,u(s))ds + \int_{t_1}^{t_2} \int_{\{z: |z|<1\}}g(s,u(s))\widetilde{N}(ds,dz) \\
&\quad+ \int_{t_1}^{t_2} \int_{\{z: |z|\ge 1\}}h(s,u(s))N(ds,dz)\bigg|^p \\
& \le c_1(p,T)(t_2 - t_1).
\end{align*}
Therefore, we only need to deal with the memory term
\begin{equation}\label{memory:regularity}
I: = \mathbb{E}\bigg|\int_0^{t_2}\kappa(t_2,s)u(s)ds - \int_0^{t_1}\kappa(t_1,s)u(s)ds\bigg|^p.
\end{equation}
Using elementary analysis and \eqref{kappa}, we can rewrite $I$ as 
\begin{align*}
I & = \mathbb{E}\bigg|\int_0^{t_1}\big(\kappa(t_2,s)-\kappa(t_1,s)\big)u(s)ds - \int_{t_1}^{t_2}\kappa(t_2,s)u(s)ds\bigg|^p \\
& \le c_2(p) \bigg\{ \mathbb{E}\bigg|\int_0^{t_1}\bigg(\frac{1}{\Gamma(1-\alpha(t_2))} - \frac{1}{\Gamma(1-\alpha(t_1))}\bigg)\frac{u(s)}{(t_1-s)^{\alpha(t_1)}}ds\bigg|^p \\
& + \mathbb{E}\bigg|\int_0^{t_1} \frac{1}{\Gamma(1-\alpha(t_2))}\bigg(\frac{1}{(t_2-s)^{\alpha(t_2)}} - \frac{1}{(t_1-s)^{\alpha(t_1)}}\bigg)u(s)ds\bigg|^p \\
& + \mathbb{E}\bigg|\int_{t_1}^{t_2} \frac{1}{\Gamma(1-\alpha(t_2))}\frac{1}{(t_2-s)^{\alpha(t_2)}}u(s)ds \bigg|^p \bigg\}\\
& =:c_2(p)(I_1 + I_2 + I_3).
\end{align*}
Since $\Gamma$ is locally Lipschitz on $(0,\infty)$ (see Lemma \ref{lem:Gamma}), we have 
\begin{equation} \label{locLip}
|\Gamma(1-\alpha(t_2))- \Gamma(1-\alpha(t_1))| \le c(T)|\alpha(t_2) - \alpha(t_1)|.
\end{equation}
It is also easy to see that
\begin{equation} \label{locbdd}
\sup_{0\le t \le T}\frac{1}{\Gamma(1-\alpha(t))} \le 1.
\end{equation}
\noindent For $I_1$, by H\"older's inequality, 
Theorem \ref{main:2},
\eqref{locbdd}, \eqref{locLip} and \textbf{(A1')}, we have 
\begin{align*}
I_1 & \le \bigg(\int_0^{t_1}\bigg|\frac{1}{\Gamma(1-\alpha(t_2))} - \frac{1}{\Gamma(1-\alpha(t_1))}\bigg|\frac{1}{(t_1-s)^{\alpha(t_1)}}ds \bigg)^{p/p'}\\
& \cdot\mathbb{E}\left[\int_0^{t_1}\bigg|\frac{1}{\Gamma(1-\alpha(t_2))} - \frac{1}{\Gamma(1-\alpha(t_1))}\bigg|\frac{|u(s)|^p}{(t_1-s)^{\alpha(t_1)}}ds\right] \\
& \le c_3(p,T)(t_2-t_1)^{p\gamma}.
\end{align*}
For $I_3$, by H\"older's inequality, 
Theorem \ref{main:2},
\eqref{locbdd} and a direct computation, we have 
\begin{align*}
I_3 & \le \bigg(\int_{t_1}^{t_2}\frac{1}{\Gamma(1-\alpha(t_2))}\frac{1}{(t_2-s)^{\alpha(t_2)}}\bigg)^{p/p'}\mathbb{E}\left[\int_{t_1}^{t_2}\frac{1}{\Gamma(1-\alpha(t_2))}\frac{|u(s)|^p}{(t_2-s)^{\alpha(t_2)}} ds\right] \\
& \le c_4(p,T)(t_2-t_1)^{(p-1)(1-\alpha^*(T))}(t_2-t_1)^{1- \alpha^*(T)} = c_4(p,T)(t_2 - t_1)^{p(1-\alpha^*(T))},
\end{align*}
where $\alpha^*(T):=\sup_{0\le t \le T}\alpha(t)$. The hard part is to estimate $I_2$, first we do a change of variables $\tilde{s}= t_1-s$ to the integral, 
\begin{align*}
I_2& = \mathbb{E}\bigg|\int_0^{t_1} \frac{1}{\Gamma(1-\alpha(t_2))}\bigg(\frac{1}{(t_2-s)^{\alpha(t_2)}} - \frac{1}{(t_1-s)^{\alpha(t_1)}}\bigg)u(s)ds\bigg|^p \\
& = \mathbb{E}\bigg|\int_0^{t_1} \frac{1}{\Gamma(1-\alpha(t_2))}\bigg(\frac{1}{(t_2-t_1+s)^{\alpha(t_2)}} - \frac{1}{s^{\alpha(t_1)}}\bigg)u(t_1 - s)ds\bigg|^p.
\end{align*}
If $t_1 \le t_2-t_1$, then by H\"older's inequality, Theorem \ref{main:2},
\eqref{locbdd} and a direct computation, we get 
\begin{align*}
I_2& \le \mathbb{E}\int_0^{t_1} \frac{1}{\Gamma(1-\alpha(t_2))}\bigg(\frac{1}{(t_2-t_1+s)^{\alpha(t_2)}} + \frac{1}{s^{\alpha(t_1)}}\bigg)|u(t_1 - s)|^pds \\
& \cdot \bigg(\int_0^{t_1} \frac{1}{\Gamma(1-\alpha(t_2))}\bigg(\frac{1}{(t_2-t_1+s)^{\alpha(t_2)}} + \frac{1}{s^{\alpha(t_1)}}\bigg)ds\bigg)^{p/p'} \\
& \le c_5(p,T)(t_2 - t_1)^{p(1-\alpha^*(T))}.
\end{align*}
If $t_1 > t_2-t_1$, we rewrite $I_2$ as 
\begin{align*}
I_2& = \mathbb{E}\bigg| \bigg(\int_0^{t_2 - t_1} + \int_{t_2 - t_1}^{t_1}\bigg) \frac{1}{\Gamma(1-\alpha(t_2))}\bigg(\frac{1}{(t_2-t_1+s)^{\alpha(t_2)}} - \frac{1}{s^{\alpha(t_1)}}\bigg)u(t_1 - s)ds\bigg|^p \\
& \le c_6(p) \bigg\{ \mathbb{E}\bigg|\int_0^{t_2 - t_1} \frac{1}{\Gamma(1-\alpha(t_2))}\bigg(\frac{1}{(t_2-t_1+s)^{\alpha(t_2)}} - \frac{1}{s^{\alpha(t_1)}}\bigg)u(t_1 - s)ds\bigg|^p \\
& + \mathbb{E}\bigg| \int_{t_2 - t_1}^{t_1} \frac{1}{\Gamma(1-\alpha(t_2))}\bigg(\frac{1}{(t_2-t_1+s)^{\alpha(t_2)}} - \frac{1}{s^{\alpha(t_2)}}\bigg)u(t_1 - s)ds\bigg|^p \\
& + \mathbb{E}\bigg| \int_{t_2 - t_1}^{t_1} \frac{1}{\Gamma(1-\alpha(t_2))}\bigg(\frac{1}{s^{\alpha(t_2)}} - \frac{1}{s^{\alpha(t_1)}}\bigg)u(t_1 - s)ds\bigg|^p
\bigg\} \\
& =: c_6(p)(I_{21}+ I_{22}+ I_{23}).
\end{align*}
$I_{21}$ can be bounded similarly as $I_2$ when $t_1 \le t_2-t_1$, so we have 
$$I_{21} \le c_7(p,T)(t_2 - t_1)^{p(1-\alpha^*(T))}.$$
For $I_{22}$, observe that for any $s$, $$\big|\frac{1}{(t_2-t_1+s)^{\alpha(t_2)}} - \frac{1}{s^{\alpha(t_2)}}\big| = \frac{1}{s^{\alpha(t_2)}} - \frac{1}{(t_2-t_1+s)^{\alpha(t_2)}}.$$
then by H\"older's inequality, Theorem \ref{main:2},
\eqref{locbdd} and a direct computation, we get
\begin{align*}
I_{22}& \le c_8(p,T) \bigg( \int_{t_2 - t_1}^{t_1}\frac{1}{s^{\alpha(t_2)}} - \frac{1}{(t_2-t_1+s)^{\alpha(t_2)}}ds\bigg)^{1+p/p'} \\
& = \frac{c_8(p,T)}{(1-\alpha(t_2))^p} \bigg( (2^{1-\alpha(t_2)}-1)(t_2-t_1)^{1-\alpha(t_2)} - (t_2^{1-\alpha(t_2)} - t_1^{1-\alpha(t_2)})\bigg)^p \\
& \le c_9(p,T)(t_2-t_1)^{p(1-\alpha^*(T))},
\end{align*}
where we used the elementary fact: for $x,y \in \mathbb{R}_+, \alpha \in [0,1]$, we have  $$|x^{\alpha} - y^{\alpha}| \le |x-y|^{\alpha}.$$
For $I_{23}$, by H\"older's inequality, Theorem \ref{main:2},
\eqref{locbdd} and a direct computation, we get
\begin{align*}
I_{23}& \le c_{10}(p,T) \bigg( \int_{t_2 - t_1}^{t_1}\bigg|\frac{1}{s^{\alpha(t_2)}} - \frac{1}{s^{\alpha(t_1)}}\bigg|ds\bigg)^p.
\end{align*}

To proceed, we need the following elementary fact: For any $0\le x \le 1$, $0\le \alpha \le 1$, it holds that
\begin{equation} \label{alpha:inequality}
|x^{\alpha}-1| \le \alpha |\log x|.
\end{equation}
\indent Without loss of generality, we assume $t_1 > 1$, otherwise only the first term will appear in the following estimate. We can also assume $\alpha(t_2) \ge \alpha(t_1)$ because the other case can be treated similarly. Applying \eqref{alpha:inequality},  we get
\begin{align*}
I_{23}& \le c_{11}(p,T) \bigg( \int_{t_2 - t_1}^{1}\bigg|\frac{1}{s^{\alpha(t_2)}} - \frac{1}{s^{\alpha(t_1)}}\bigg|ds\bigg)^p +  c_{11}(p,T) \bigg( \int_{1}^{t_1}\bigg|\frac{1}{s^{\alpha(t_2)}} - \frac{1}{s^{\alpha(t_1)}}\bigg|ds\bigg)^p \\
& \le c_{11}(p,T) \bigg( \int_{t_2 - t_1}^{1}s^{-\alpha(t_2)}\bigg|s^{\alpha(t_2) - \alpha(t_1)} - 1\bigg|ds\bigg)^p\\
&\quad  +  c_{11}(p,T) \bigg( \int_{1}^{t_1}s^{-\alpha(t_1)}\bigg|\frac{1}{s^{\alpha(t_2) -\alpha(t_1)}} -1\bigg|ds\bigg)^p \\
& \le c_{12}(p,T)(t_2-t_1)^{p\gamma} \bigg\{\bigg( \int_{t_2 - t_1}^{1}s^{-\alpha(t_2)}|\log s|ds\bigg)^p + \bigg( \int_{1}^{t_1}s^{-\alpha(t_1)}|\log s|ds\bigg)^p \bigg\} \\
& \le c_{13}(p,T)(t_2-t_1)^{p\gamma}.
\end{align*}
Note that for the third inequality, we used \eqref{alpha:inequality}. Combining all the estimates above, we have 
\begin{equation}\label{finalregularity}
\mathbb{E}|u(t_2)-u(t_1)|^p \le C_3(p,T)(t_2-t_1)^{C_4(p,T)},
\end{equation}
where $C_4(p,T): = \min \{1, p\gamma, p(1-\alpha^*(T))\}$. \\
\indent Now for the case $p \ge 2$, the estimates for the memory term are the same. Since we assume the $L^2$ linear growth condition for $g$, our result follows easily by applying \eqref{ge2}. In fact, the result is exactly of the form \eqref{finalregularity}.
\qed 
\begin{remark}
	If we replace the L\'evy noise by a white noise, then for any $p\ge 2$, using the well-known Burkholder-Davis-Gundy inequality, we can show that the solution $u(t)= u_0 + \int_{0}^{t}\kappa(t,s)u(s)ds +\int_{0}^{t}f\big(s,u(s)\big)ds + \int_{0}^{t}g(s,u(s))dB_s$ satisfies 
	\begin{align*}
	\mathbb{E}|u(t_2)-u(t_1)|^p \le c_1(p,T) (t_2-t_1)^{c_2(p,T)},
	\end{align*}
	where $$c_2(p,T):= \min\{2-\frac{2}{p}, p\gamma, p(1-\alpha^*(T))\}$$
	Then from the well-known Kolmogorov continuity theorem, if we take $p$ large enough such that $c_2(p,T)>1$, then the above solution has a H\"older continuous version.(Recall that we need the power of the time to be greater than 1)  \\
	\indent However, in our case, if one of the jump terms 
	$$\int_{\{z: |z|<1\}}g(s,u(s))\widetilde{N}(ds,dz),\ \int_{t_1}^{t_2} \int_{\{z: |z|\ge 1\}}h(s,u(s))N(ds,dz)$$
	appears, then from our moment inequalities, the power in \eqref{finalregularity} is
	$$C_4(p,T): = \min \{1, p\gamma, p(1-\alpha^*(T))\} \le 1$$
	Thus we can not expect the solution to have a continuous version, which is natural because we have random jumps.
\end{remark}
\section{Further discussions}
\subsection{General kernel function $\kappa(t,s)$}
Though we are dealing with a special kernel function, which is given as \eqref{kappa}, our method definitely applies to more general kernel functions, as in \cite{Zhang}. Indeed, to establish Theorem \ref{main:1}, if we take a closer look at the proof, the essential ingredients 
are: the Gronwall type inequality in Lemma \ref{Gronwall}, and also Lemma \ref{scale}.\\
\indent Suppose the kernel function $\kappa(t,s)$ in \eqref{integral form} is a general one, i.e., $\kappa: \{(s,t): 0\le s< t\} \rightarrow (0,\infty) $ is a measurable function. We want $\kappa(t,s)$ to satisfy Lemma \ref{Gronwall} and Lemma \ref{scale}. Indeed, for the Gronwall type inequality, we need the following: For any $T>0$, there exists a constant $c_1(T)>0$ and another kernel function $\widetilde{\kappa}: \{(s,t): 0\le s< t\}\rightarrow (0,\infty)$ such that 
\begin{align*}
\kappa(t,s) \le c_1(T) \widetilde{\kappa}(t,s),\ 0\le s < t \le T 
\end{align*}
and $\widetilde{\kappa}(t,s)$ satisfies 
\begin{equation} \label{general kernel:1}
\begin{aligned}
& \sup_{0\le t\le T}\int_0^t\widetilde{\kappa}(t,s)ds < +\infty,\ \forall T>0, \\
& \overline{\lim_{\varepsilon \rightarrow 0}}\sup_{0\le t \le T}\int_t^{t+\varepsilon} \widetilde{\kappa}(t+ \varepsilon,s)ds < 1,\ \forall T > 0.
\end{aligned}
\end{equation}
We can obtain (see \cite{Zhang} for details) the following result: 
\begin{lem}
	Define $$
	\begin{cases}
	r_1(t,s): = \widetilde{\kappa}(t,s) \\
	r_n(t,s):= \int_s^t \widetilde{\kappa}(t,u)r_{n-1}(u,s)du,\ n\ge 1.
	\end{cases}$$
	Then $\forall T>0$, there exist $C_{10}(T)>0$ and $C_{11}(T)\in (0, 1)$ such that 
	\begin{equation}
	\begin{aligned}
	\sup_{0\le t\le T}\int_0^t r_n(t,s)ds \le C_{10}(T)nC_{11}(T)^n\quad \forall n\ge1.
	\end{aligned}
	\end{equation}
	Moreover, 
	\begin{equation} \label{resolvent}
	r(t,s):= \sum_{n\ge1}r_n(t,s)
	\end{equation}
	satisfies 
	$
	r(t,s) - \widetilde{\kappa}(t,s) = \int_s^t \widetilde{\kappa}(t,u)r(u,s)du = \int_s^t r(t,u)\widetilde{\kappa}(u,s)du.
	$
\end{lem} 
Using the above result, the following Volterra type Gronwall inequality was proved in \cite{Zhang}:
\begin{lem} \label{general Gronwall}
	Suppose $\varphi(\cdot), \varphi_0(\cdot)$ are two non-negative locally bounded functions defined on $[0,\infty)$ such that for any $0\le a < b$, 
	$$\varphi(t) \le \varphi_0(t) + \int_a^t \widetilde{\kappa}(t,s)\varphi(s)ds,\ \forall t \in [a,b).$$
	Then we have 
	$$\varphi(t) \le \varphi_0(t)+ \int_a^t r(t,s)\varphi_0(s)ds,\ \forall t \in [a,b).$$
	where $r(t,s)$ is defined as in \eqref{resolvent}. 
\end{lem}
Apart from Gronwall type inequality, we also need a counterpart to Lemma \ref{scale}. Indeed, the following lemma gives a sufficient condition: 
\begin{lem} \label{general scale}
	If for any $u \in (0,1)$, the function 
	\begin{equation} \label{increase}
	t \mapsto t\widetilde{\kappa}(t,tu)
	\end{equation}
	is an increasing function. Then for any $\theta: [0,+\infty) \rightarrow \mathbb{R}$ locally bounded and $\alpha^* \in [0,1)$, we have for all $T>0$, 
	\begin{equation*}
	\sup_{0\leq t\leq T}\int_0^t \widetilde{\kappa}(t,s) |\theta(s)|ds \le \int_{0}^{T}\sup_{0\leq u\leq t}|\theta(u)|\widetilde{\kappa}(T,t) dt.
	\end{equation*}
\end{lem}
\begin{proof}
	We can use exactly the same change of variables technique as in the proof of Lemma \ref{scale} to complete the proof.
\end{proof}
Then using the above two lemmas, we can show that Theorem \ref{main:1} and Theorem \ref{main:2} hold  under the conditions \eqref{general kernel:1}, \eqref{increase}. However, for the regularity Theorem \ref{main:3}, we need the special form \eqref{kappa} of the kernel function. 
\subsection{Further questions}
To get the well-posedness of the solution, can the Lipschitz condition \textbf{(A2)} be weakened to an integrability condition as in the case of ordinary SDE, e.g. \cite{XieZhang}? Here we remark that the non-Lipschitz condition in \cite{Wangz} applies here, and what we want is the condition as in \cite{XieZhang}.
In our proof, we essentially used the Lipschitz condition for the small jump coefficient twice. 

First, in our approximating procedure, we need the Lipschitz condition to bound the error so that it decays fast enough. 
In our case, it decays slower than the case of ordinary SDE, but still summable, see \eqref{est3}. Also, it works for general kernel function using Lemma \ref{general Gronwall}. In the case of SDE \cite{XieZhang}, they do not need the Lipschitz condition to bound the error. Indeed, they need an a priori estimate of the Kolmogorov forward equation.
Then by applying Zvonkin's transform, they can bound the error. However, in our case, we have the memory term and we do not have a corresponding Kolmogorov forward equation.

Second, we need the Lipschitz condition to make sure that we can add the large jump term. Recall that in case of ordinary SDE, once we get the well-posedness of the solution without large jumps, we can automatically add the large jump term. The reason is that the solution is a Markov process and the transition function is measurable with respect to the initial point. However, in our case, the stochastic process is non-Markovian. To apply the interlacing procedure to add the large jump term, we use the stopping time argument, see the proof of Proposition \ref{stronger result}, where we essentially used the Lipschitz condtition. \\

The proof of Theorem \ref{main:3} shows that, if the fractional order function $\alpha(\cdot)$ does not have good regularity property, then the regularity the solution will be worse. 
What about some other properties if we weaken the regularity of $\alpha(\cdot)$? For example, what can one say about long time behavior of the solution \cite{Jacquier, Zhang}, density of the solution, etc. \cite{Besalu2}.
\section*{Acknowledgements}

We thank the referees for the constructive and helpful comments, which greatly
improved the quality of this paper. This work was funded in part by the Army Research Office (ARO) MURI Grant W911NF-15-1-0562, by the National Science Foundation under Grants DMS-1620194 and DMS-2012291, and by the Simons Foundation (\#429343).

\end{document}